\documentclass[reqno]{amsart}

\usepackage{color}
\usepackage[dvipsnames]{xcolor}

\usepackage{ifpdf}
\ifpdf 
\usepackage[pdftex]{graphicx}   
\usepackage[pdftex,     
plainpages=false,   
breaklinks=true,    
colorlinks=true,
linkcolor=red,
citecolor=green,
pdftitle=My Document
pdfauthor=My Good Self
]{hyperref} 
\else 
\usepackage{graphicx}       
\usepackage{hyperref}       
\fi 

\usepackage{subfig}
\usepackage{glossaries}
\usepackage{glossary-mcols}

\usepackage{aurical}
\usepackage{amsfonts,amsmath}	
\usepackage{amssymb}
\usepackage{verbatim}
\usepackage{amsopn}
\usepackage[english]{babel}
\usepackage{amsthm}
\usepackage{enumerate}
\usepackage{mathrsfs}	
\usepackage{mathtools}
\usepackage{esint}
\usepackage{todonotes}

\usepackage{bm}
\usepackage{bbm}
\usepackage{caption}

\theoremstyle{definition} \newtheorem{definition}{Definition}[section]
\theoremstyle{definition} \newtheorem{remark}[definition]{Remark}
\theoremstyle{plain} \newtheorem{lemma}[definition]{Lemma}
\theoremstyle{plain} \newtheorem{proposition}[definition]{Proposition}
\theoremstyle{plain} \newtheorem{theorem}[definition]{Theorem}
\theoremstyle{plain} \newtheorem{corollary}[definition]{Corollary}
\theoremstyle{definition} 
\theoremstyle{plain} 
\theoremstyle{definition}

\DeclareMathOperator{\BV}{BV}

\DeclareMathOperator{\dist}{dist}

\DeclareMathOperator{\Lip}{Lip}

\DeclareMathOperator{\Per}{\mathsf{Per}}

\newcommand{\R}{\mathbb{R}}

\newcommand{\N}{\mathbb{N}}

\newcommand{\TV}{\text{\rm Tot.Var.}}
\newcommand{\Int}{\mathrm{Int}}

\newcommand{\e}{\varepsilon}

\newcommand{\loc}{\text{\rm loc}}

\renewcommand{\L}{\mathscr L}

\newcommand{\T}{\mathbb{T}}

\renewcommand{\div}{\mathrm{div}}
\renewcommand{\L}{\mathscr L}
\renewcommand{\H}{\mathscr H}

\numberwithin{equation}{section} 

\theoremstyle{plain} \newtheorem*{theorem*}{Theorem}
\theoremstyle{plain} 
\theoremstyle{plain} \newtheorem*{mthm*}{Main Theorem}
\theoremstyle{plain} \newtheorem*{conjecture*}{Conjecture}
\theoremstyle{plain} 
\theoremstyle{plain} \newtheorem*{problem*}{Problem}

\title[Flow of 2D autonomous BV divergence-free vector fields]{Regularity estimates for the flow of BV autonomous divergence-free vector fields in $\mathbb{R}^2$}

\author[P.~Bonicatto]{Paolo Bonicatto}
\address{P.B. Departement Mathematik und Informatik,
Universit\"at Basel, Spiegelgasse 1, CH-4051 Basel, Switzerland.}
\email{paolo.bonicatto@unibas.ch}

\author[E.~Marconi]{Elio Marconi}
\address{E.M. Departement Mathematik und Informatik,
Universit\"at Basel, Spiegelgasse 1, CH-4051 Basel, Switzerland.}
\email{elio.marconi@unibas.ch}

\thanks{MSC class: 34C11 (primary), 35L45, 37C10}

\begin{document}
	\maketitle
	
	\begin{abstract}
	We consider the regular Lagrangian flow $X$ associated to a bounded divergence-free vector field $b$ with bounded variation. 
	We prove a Lusin-Lipschitz regularity result for $X$ and we show that the Lipschitz constant grows at most linearly in time.
	As a consequence we deduce that both geometric and analytical mixing have a lower bound of order $1/t$ as $t\to \infty$.
	\end{abstract}

\section{Introduction}
We consider a bounded vector field $b \in L^\infty([0,+\infty)\times \R^d, \R^d)$ and we introduce the following notion of flow.
\begin{definition}
We say that $X:[0,+\infty)\times \R^d\to \R^d$ is a \emph{regular Lagrangian flow} of the vector field $b$ if
	\begin{enumerate}
	\item for $\mathscr L^d$- a.e. $x\in \R^d$ the map $t\mapsto X(t,x)$ is Lipschitz, $X(0,x)=x$ and for $\mathscr L^1$-a.e. $t>0$ it holds $\partial_tX(t,x)=b(t,X(t,x))$;
	\item for every $t\ge 0$ it holds
	\begin{equation*}
	X(t,\cdot)_\sharp \mathscr L^d \le L\mathscr L^d,
	\end{equation*}
	for some $L>0$.
	\end{enumerate}
\end{definition}
Regular Lagrangian flows have been introduced in \cite{DPL_transport} in the setting of Sobolev vector fields with bounded divergence. 
The authors deduced existence and uniqueness of regular Lagrangian flows by proving the well-posedness in the class of bounded weak solutions for the strictly related 
transport equation driven by $b$:
\begin{equation}\label{E_transport}
\begin{cases}
\partial_t u + b \cdot \nabla u =0 & \mbox{for }(t,x)\in [0,+\infty)\times \R^d, \\ 
u(0,\cdot)=u_0 & \mbox{in }\R^d.
\end{cases}
\end{equation}
A purely lagrangian proof of the well-posedness for regular Lagrangian flows associated to Sobolev vector fields has been provided in \cite{CDL_DiPerna-Lions}.
Among the advantages of the proof in \cite{CDL_DiPerna-Lions} there is the following quantitative regularity estimate on the regular Lagrangian flow.
\begin{theorem}\label{T_CDL}
Let $X:[0,+\infty)\times \R^d\to \R^d$ be the regular Lagrangian flow associated to a vector field $b\in L^1_\loc([0,+\infty);W^{1,p}(\R^d))$ for some $p>1$. Then, for every $\e>0$ and every $R>0$ there exists $B\subset B_R(0)$ such that
$\L^d(B)\le \e$ and for every $t \ge 0$ it holds
\begin{equation*}
\Lip(X(t,\cdot)\llcorner (B_R(0)\setminus B))\le \exp \left( \frac{C}{\e^{1/p}}\int_0^t\|D_xb(t,\cdot)\|_{L^p}dt\right),
\end{equation*}
where $C=C(d,p,R,L)>0$. In particular if $\exists M>0$ such that for every $t>0$ it holds $\|Db(t,\cdot)\|_{L^p}\le M$, then
\begin{equation*}
\Lip(X(t,\cdot)\llcorner (B_R(0)\setminus B))\le \exp \left( \frac{CM}{\e^{1/p}}t\right).
\end{equation*}
\end{theorem}

The Eulerian approach of \cite{DPL_transport} has been extended to the case of vector fields $b \in L^1([0,T);BV(\R^d;\R^d))$ with bounded divergence in \cite{Ambrosio_transport}. 
In particular the result in \cite{Ambrosio_transport} covers the setting of this work: we consider bounded planar autonomous vector fields $b \in \BV(\R^2,\R^2)$
with compact support and such that $\div\, b=0$.
The case of bounded planar autonomous vector fields has been exhaustively studied in the series of papers \cite{ABC2,ABC3,ABC1}, 
where the authors provided sufficient and necessary conditions for the uniqueness of \eqref{E_transport} in the class of bounded weak solutions.
The very precise analysis in this setting is allowed by the Hamiltonian structure that these vector fields have: for any vector field $b$ as above there exists an Hamiltonian
$H\in \Lip(\R^2)$ such that $b=\nabla^\perp H$. Formally the Hamiltonian is constant along the flow of $b$ and the uniqueness problem for \eqref{E_transport} is splitted into a family of
one dimensional problems on the level sets of the Hamiltonian $H$. 

However quantitative regularity results analogous to Theorem \ref{T_CDL} in the $\BV$ setting are still unknown, even in the simpler setting of this work.
Our main result is the following.

\begin{theorem}\label{T_main}
Let $b\in \BV(\R^2,\R^2)$ be a bounded autonomous vector field with compact support and $\div\, b=0$ and denote by $X$ the associated regular Lagrangian flow.
	Then for every $\e>0$ there exists $C=C(\e,b)>0$ and $B\subset \R^2$ with $\L^2(B)\le \e$ such that for every $t\ge 0$ it holds
	\begin{equation*}
	\Lip (X(t)\llcorner (\R^2\setminus B))\le C(1+t).
	\end{equation*}
\end{theorem}
We point out that the Lipschitz constant depends on $b$ and not only on its total variation. Moreover we cannot make the dependence of $C$ on $\e$ explicit.
On the other hand we recover the linear in time growth of the Lipschitz constant, which is optimal and peculiar of the autonomous case.

As observed in \cite{CDL_DiPerna-Lions}, it easily follows from Theorem \ref{T_main} that for every $t\ge 0$ the regular Lagrangian flow $X(t,\cdot)$ is approximately differentiable for 
$\L^2$-a.e. $x\in \R^2$. We recall that a Borel map $f:\R^2\to \R^2$ is approximately differentiable at $x\in \R^2$ if there exists a linear map $L:\R^2 \to \R^2$ such that
the difference quotients
\begin{equation*}
y\to \frac{f(x+\e y)-f(x)}{\e}
\end{equation*}
locally converge in measure as $\e\to 0$ to $Ly$. We refer to \cite{CDL_DiPerna-Lions} for further details about this aspect.

As in \cite{ABC1}, the proof of Theorem \ref{T_main} relies on the Hamiltonian structure of $b$ and in particular on the structure of the level sets of a Lipschitz Hamiltonian $H$ on the plane
such that $\nabla H\in \BV(\R^2,\R^2)$. In \cite{BKK_Sard} it is proved that for $\L^1$-a.e. $h\in H(\R^2)$ the level set $H^{-1}(h)$ consists of finitely many disjoint cycles.
In particular the trajectories of the flow are contained into these cycles.
This rigidity allows to compare the evolution of two points without appealing to the basic inequality
\begin{equation*}
|b(x)-b(y)|\le (M(x)+M(y))|x-y|,
\end{equation*}
where $M$ denotes the maximal function of $|D_xb|$. This estimate is crucial in \cite{CDL_DiPerna-Lions} and the failure of the $L^1$-continuity of the maximal operator is what fails in the proof of Theorem \ref{T_CDL} for $p=1$.
On the other hand, since our estimate depends on the structure of $H$, we cannot prove that the constant $C$ in Theorem \ref{T_main} depends on $b$ only through its total
variation.

It is well-known that quantitative regularity results on the flow $X$ associated to $b$ imply lower bounds on the mixing scale of passive scalars driven by $b$ 
through \eqref{E_transport}. In this work we will consider the two notions of geometric and analytical mixing introduced respectively in \cite{Bressan_mix_conj} and \cite{MMP_mixing,LTD_mixing}.
In this introduction let us consider the periodic setting, the proper notions for the planar case will be introduced in Section \ref{S_mixing}: we denote by $\T^2$ the two dimensional torus.
\begin{definition}\label{D_mix}
Let $k \in (0,1/2)$ and $A\subset \T^2$ such that $\L^2(A)=1/2$. We say that the flow $X$ \emph{geometrically mixes $A$ up to scale $\e$ at time $t$} if for every $x\in \T^2$ it holds
\begin{equation*}
k\L^2(B_\e(x))\le \frac{\L^2(B_\e(x)\cap X(t,A))}{\L^2(B_\e(x))}\le (1-k)\L^2(B_\e(x)).
\end{equation*}
We say moreover that $X$ \emph{analytically mixes $A$ up to scale $\e$ at time $t$} if 
\begin{equation*}
\|\chi_{X(t,A)}-\chi_{X(t,A^c)}\|_{\dot H^{-1}(\T^2)}\le \e,
\end{equation*}
where $\chi_E$ denotes the indicator function of the set $E\subset \T^2$ and
	\begin{equation*}
	\|u\|_{\dot H^{-1}(\T^2)}:=\sup \left\{\int_{\T^2}u\phi dx: \|\nabla \phi\|_{L^2}\le 1 \right\}.
	\end{equation*}

\end{definition}
It is known that the two notions are not equivalent: the relation between them has been extensively studied in \cite{Zillinger_mix_g_a}.

The flow $X$ is said to be nearly incompressible if there exists $k'>0$ such that for every $t\in [0,T]$ and every $\Omega \subset \T^2$
\begin{equation}
\frac{1}{k'}\L^d(\Omega)\le \L^d(X(t,\Omega))\le k'\L^d(\Omega).
\end{equation}
In \cite{CDL_DiPerna-Lions} it has been proved that for $p>1$ there exists a constant $C(k,k',p)$ such that if $X$ is a nearly incompressible flow of a smooth time-dependent vector field that geometrically 
mixes the set $A=[0,1)\times [0,1/2)\subset \T^2$ at time $t=1$ then
\begin{equation*}
\int_0^1\|D_xb(t)\|_{L^p}dt\ge C|\log \e|.
\end{equation*}

The same statement for $p=1$ is the well-known Bressan's mixing conjecture \cite{Bressan_mix_conj}.
From the eulerian point of view, uniqueness for $\BV$ nearly incompressible vector fields has been recently established in \cite{BB_Bressan}. 
The argument in  \cite{CDL_DiPerna-Lions} has been extended in
\cite{IKX_mixing, Seis_mixing} in order to deal with more general initial data and to the analytical mixing scale as well. Following \cite{IKX_mixing} with minor changes we get from Theorem \ref{T_main}
the following result.

\begin{theorem}\label{T_mix}
Let $b\in \BV(\R^2,\R^2)$ be bounded and divergence-free and denote by $X$ the associated regular Lagrangian flow. Let $k\in (0,1/2)$ be as in Definition \ref{D_mix} and let $A\subset B_1$ be a measurable set such that 
$\L^2(A)=\L^2(B_1)/2$. Then there exists a constant $c_g=c_g(A,b,k)>0$ such that if $X$ geometrically mixes $A$ at scale $\delta$ at time $t$ on $B_1$, then
\begin{equation*}
\delta\ge \frac{c_g}{1+t}.
\end{equation*}
Moreover there exists $c_a=c_a(A,b)$ such that for every $t\ge 0$
\begin{equation*}
\|\chi_{X(t,A)}-\chi_{X(t,A^c)}\|_{\dot H^{-1}(B_1)}\ge \frac{c_a}{1+t}.
\end{equation*}
\end{theorem}
We observe moreover that the dependence of $c_g$ and $c_a$ on the set $A$ can be made explicit assuming that $\Per(A,B_1)<\infty$ and we refer to Section \ref{S_mixing} 
for further details.
We finally notice that examples of mixers in this setting have been provided in \cite{Zillinger_circular,CLS_mixing} and for shear flows in \cite{BCZ_enhanced}.

\subsection{Structure of the paper}
Section \ref{S_prel} is devoted to preliminaries: we recall the structure of level sets of Lipschitz functions with compact support and gradient with bounded variation from \cite{BKK_Sard}, we introduce the notion of regular Lagrangian flow and we set a few facts about simple curves in $\R^2$.
In Section \ref{S_cov} we introduce a way to map a cycle into another cycle. This will be useful in Section \ref{S_Lip} to compare the evolution of two points 
belonging to different level sets of $H$ and eventually to prove Theorem \ref{T_main}. 
The lower bounds on the geometrical and analytical mixing stated in Theorem \ref{T_mix} are obtained in Section \ref{S_mixing}.

\subsection*{Acknowledgements}
The authors acknowledge ERC Starting Grant 676675 FLIRT. Moreover the authors acknowledge Gianluca Crippa for proposing the problem and for a careful reading of the manuscript and Guido De Philippis for a preliminary discussion.
	
\section{Preliminaries and setting}\label{S_prel}
\subsection{Structure of $BV$ divergence free vector fields in $\R^2$}
	We consider a bounded vector field $b\in \BV(\R^2,\R^2)$ with compact support and $\div \, b=0$.  
	We recall a first structure theorem for $\BV$ vector fields from \cite{AFP_book}.
	\begin{theorem}
	Let $b \in \BV(\R^2,\R^2)$. Then there exists a partition of the plane $\R^2=S_b\cup J_b\cup D_b$ for which the following properties hold:
	\begin{enumerate}
	\item $\mathscr H^1(S_b)=0$;
	\item $J_b$ is contained in the union of the images of at most countably many Lipschitz curves and there exist Borel functions
	$\nu:J_b\to \mathbb{S}^1$, $b^-,b^+:J_b\to \R^d$ such that for every $\bar x\in J_b$
	\begin{equation*}
	\fint_{B_r^+(\bar x)}|b(x)-b^+(\bar x)|dx =o(1) \qquad \mbox{and} \qquad \fint_{B_r^-(\bar x)}|b(x)-b^-(\bar x)|dx =o(1) \qquad \mbox{as }r\to 0,
	\end{equation*}
	where $B^+_r(\bar x):=\{x\in B_r:(x-\bar x)\cdot \nu>0\}$ and $B^-_r(\bar x):=\{x\in B_r:(x-\bar x)\cdot \nu<0\}$;
	\item for every $\bar x\in D_b$ there exists $\tilde b(\bar x)$ such that
	\begin{equation*}
	\fint_{B_r(\bar x)}|b(x)-\tilde b(\bar x)|dx=o(1) \qquad \mbox{as }r\to 0.
	\end{equation*}
	\end{enumerate}
	\end{theorem}
	In the following we will always assume that $b=\tilde b$ on $D_b$ and that $b=(b^++b^-)/2$ on $J_b$.
	
	Since $\div b=0$ and $b$ is bounded with compact support, there exists an Hamiltonian $H\in \Lip_c(\R^2)$ such that $b=\nabla^\perp H$. 
	In the following theorem from \cite{BKK_Sard} we describe the structure of the level sets of $H$. Let us first introduce some notation:
	by \emph{cycle} we mean a set which is homeomorphic to the unit circle $\mathbb{S}^1 \subset \R^2$. In the following we will consider parametrizations of cycles,
	i.e. $\gamma:[0,l]^*\to \R^2$, where $[0,l]^*$ denotes the set $[0,l]^*$, where the endpoints are identified.
	By $[s_1,s_2]\subset [0,l]^*$ we mean $[s_1,s_2]$ if $s_1\le s_2$ and $[s_1,l] \cup [0,s_2]$ if $s_1>s_2$.
	Accordingly we set
	\begin{equation*}
	\int_{s_1}^{s_2}f=\int_{[s_1,s_2]}f \quad\mbox{if }s_1\le s_2 \qquad \mbox{and}\qquad\int_{s_1}^{s_2}f=\int_{[0,s_2]\cup[s_1,l]}f \quad \mbox{if }s_2 < s_1.
	\end{equation*}
	
	\begin{theorem}\label{T_Bourgain}
	Let $H\in \Lip_c(\R^2)$ be such that $b:=\nabla^\perp H \in \BV(\R^2,\R^2)$.
	Then there exists $R\subset \R$ such that $\L^1(H(\R^2)\setminus R)=0$ and for every $r\in R$ the following properties hold: 
	\begin{enumerate}
	\item $\mathscr H^1(H^{-1}(h)) < \infty$;
	\item there exist $N=N(h)\in \N$ and $C_h^i$ disjoint cycles such that 
	\begin{equation*}
	H^{-1}(h)=\bigcup_{i=1}^N C_h^i.
	\end{equation*}
	Notice that by Point (1) we can parametrize the cycles by simple curves $\gamma_h^i:[0,l_h^i]^*\to \R^2$ such that 
	$|\dot\gamma_h^i(s)|=1$ for $\mathscr L^1$-a.e. $s \in [0,l_h^i]^*$ and $\gamma_h^i([0,l_h^i]^*)=C_h^i$. In the following we will consider parametrizations with these properties.
	\item for every $i=1,\ldots, N$ the \emph{integral curvature} of $C_h^i$ 
	\begin{equation*}
	\TV_{[0,l_h^i]^*}(\gamma_h^i)' <\infty.
	\end{equation*}
	By total variation for a function $v:[0,l]^*\to \R^2$ we mean
	\begin{equation*}
	\TV_{[0,l]^*}v=\sup_{0\le a_0<...<a_k\le l}\sum_{i=1}^k|v(a_{i+1})-v(a_i)|,
	\end{equation*}
	where we set $a_{k+1}=a_0$ and $(a_i)_{i=0}^k$ are Lebesgue points of $v$;
	\item $H^{-1}(h)\cap S_b=\emptyset$ and $H^{-1}(h)\cap J_b$ is at most countable; moreover for every $x\in H^{-1}(h)\cap J_b$
	\begin{equation*}
	b^+(x)\cdot \nu(x)\ne 0 \ne b^-(x)\cdot \nu(x).
	\end{equation*}
	\end{enumerate}
	Moreover for every $\e>0$ there exist $F_\e \subset H(\R^2)$ and $c_S>0$ such that $\mathscr L^1(H(\R^2)\setminus F_\e)<\e$ and for every $h\in F_\e$ the following additional
	condition is satisfied: 
	\begin{enumerate}
	\item[(5)] for every $x \in H^{-1}(h)\cap D_b$ it holds
	\begin{equation*}
	|b(x)|\ge c_S.
	\end{equation*}
	\end{enumerate}
	\end{theorem}
	
	The following decomposition theorem has been obtained for Hamiltonians in $\BV(\R^2)$ with minor modifications in \cite{BT_monotone}. 
	Here we state it in the setting of Lipschitz Hamiltonians.
	\begin{theorem}\label{T_ABC}
	Let $H\in \Lip_c(\R^2)$. Then there exist $\{H_i\}_{i=1}^\infty$ with $H_i\in \Lip_c(\R^2)$ and a pairwise disjoint family $\{A_i\}_{i\in \N}$ with $A_i\subset \R^2$
	such that the following properties hold:
	\begin{enumerate}
	\item $H=\sum_{i=1}^\infty H_i$;
	\item the function $H_i$ is \emph{monotone} for every $i\in \N$, i.e. every level set of $H$ is connected;
	\item $\nabla H_i$ is concentrated on $A_i$ for every $i\in \N$.
	\end{enumerate}
	\end{theorem}
	
	By Theorem \ref{T_Bourgain} and by inspection of the proof in \cite{BT_monotone} we get the following additional property.
	\begin{corollary}\label{C_ABC}
	Let $H\in \Lip_c(\R^2)$ be such that $b=\nabla^\perp H\in \BV(\R^2,\R^2)$. 
	Then there exists $R'\subset H(\R^2)$ such that $\L^1(H(\R^2)\setminus R')=0$ and for every $h\in R'$ Properties (1) - (4) in Theorem \ref{T_Bourgain} hold and 
	for every cycle $C^j_h\in H^{-1}(h)$ there exist unique $\bar i(j,h)\in \N$ and $\bar h_i(j,h)\in H_i(\R^2)$ such that $C^j_h=H_i^{-1}(\bar h_i(j,h))$ and 
	$b(x)  = \nabla^\perp H_i(x)$ for $\H^1$-a.e. $x\in C^j_h$ . In particular $R'$ and the sets $A_i$ in Theorem \ref{T_ABC} can be chosen in the following way:
	\begin{equation*}
	A_i:= H_i^{-1}\left( \bigcup_{h\in R', j=1,\ldots,N(h)}\{\bar h_i(j,h)\}\right).
	\end{equation*}
	\end{corollary}

\subsection{The continuity equation and regular Lagrangian flows}
	Under the same assumptions on the vector field $b$ as in the previous section we introduce the notion of regular Lagrangian flow associated to $b$.
	\begin{definition}
	We say that $X:[0,+\infty)\times \R^2\to \R^2$ is a \emph{regular Lagrangian flow} of the vector field $b$ if
	\begin{enumerate}
	\item for $\mathscr L^2$- a.e. $x\in \R^2$ the map $t\mapsto X(t,x)$ is Lipschitz, $X(0,x)=x$ and for $\mathscr L^1$-a.e. $t>0$ it holds $\partial_tX(t,x)=b(X(t,x))$;
	\item for every $t\ge 0$ it holds
	\begin{equation*}
	X(t,\cdot)_\sharp \mathscr L^2 = \mathscr L^2.
	\end{equation*}
	\end{enumerate}
	\end{definition}
	It is well-known that this notion is strictly related to the continuity equation driven by $b$;
	we consider the Cauchy problem for $(t,x)\in [0,+\infty)\times \R^2$ with bounded initial datum $u_0$:
	\begin{equation}\label{E_CE_Cauchy}
	\begin{cases}
	\partial_t u +\div_x (ub)=0, \\
	u(0,\cdot)=u_0.
	\end{cases}
	\end{equation}
	
	Existence and uniqueness for regular Lagrangian flows associated to $b$ and for bounded weak solutions to \eqref{E_CE_Cauchy} hold in this setting: we refer to \cite{ABC1} for an
	exhaustive discussion about these questions for bounded divergence free vector fields in $\R^2$.
	\begin{theorem}\label{T_uniqueness}
	Let $b=\nabla^\perp H \in \BV(\R^2,\R^2)$ with $H\in \Lip_c(\R^2)$. Then there exists a regular Lagrangian flow $X$ of $b$ and for any other regular Lagrangian flow $\tilde X$
	of $b$ it holds $X(t,x)=\tilde X(t,x)$ for  $\L^2$-a.e. $x\in \R^2$ and every $t\ge 0$. 
	Moreover for $\L^2$-a.e. $x\in \R^2$ and every $t\ge 0$ it holds
	\begin{equation*}
	H(X(t,x))=H(x)
	\end{equation*}
	and for every $u_0\in L^\infty(\R^2)$ the unique bounded weak solution to \eqref{E_CE_Cauchy} is given by
	\begin{equation*}
	u(t)\L^2 = X(t)_\sharp \left(u_0\L^2\right).
	\end{equation*}
	\end{theorem}
	
	In the following remark we fix a representative of the regular Lagrangian flow $X$ of $b$ for future references.
	\begin{remark}\label{R_representative}
	From Theorem \ref{T_uniqueness} and Theorem \ref{T_Bourgain} it follows that there exists a representative of the flow $\tilde X$ such that
	for every $h\in R$ and \emph{every} $x\in C^j_h$ for some cycle $C^j_h \in H^{-1}(h)$ the flow $t\mapsto\tilde X(t,x)$ is the unique characteristic curve of $b$ with image
	contained in $C^j_h$.
	It is not hard to check that if $C^j_h$ is parametrized by $\gamma^j_h$ as in Theorem \ref{T_Bourgain}, then $\tilde X(t,x)=\gamma_h^j(\bar s)$, where
	$\bar s\in [0,l_{\gamma^j_h}]^*$ is uniquely determined by
	\begin{equation*}
	t-\int_{(\gamma^j_h)^{-1}(x)}^{\bar s}\frac{1}{|b(\gamma^j_h(s))|}ds = k\int_0^{l_{\gamma^j_h}}\frac{1}{|b(\gamma^j_h(s))|}ds,
	\end{equation*}
	for some $k\in \N$.
	\end{remark}

\subsection{Closed curves with finite curvature in $\R^2$}
	In this section we consider cycles with finite length and parametrized by one-to-one curves $\gamma:[0,l_\gamma]^*\to \R^2$ with 
	$|\gamma'|(s)=1$ for $\L^1$-a.e. $s \in [0,l_\gamma]^*$ and $\TV_{[0,l_\gamma]^*}\gamma'< \infty$. We refer to these $\gamma$ as \emph{simple curves with finite turn}.
				
	\begin{lemma}\label{L_Lip_inverse}
	Let $\gamma:[0,l_\gamma]^*\to \R^2$ be a simple curve as above without cusps, i.e. there exist no $\bar s \in [0,l_\gamma]^*$ such that
	\begin{equation*}
	\lim_{s\to \bar s^-}\gamma'(s)= - \lim_{s\to \bar s^+}\gamma'(s).
	\end{equation*}
	Then the inverse of $\gamma$ defined on $\gamma([0,l_\gamma]^*)\subset \R^2$ with values in $[0,l_\gamma]^*$ is Lipschitz.
	\end{lemma}
	\begin{proof}
	Denote by $\mu_\gamma$ the total variation of the measure $\gamma''$:
	since $\gamma$ has no cusps there exists $\e\in (0,2]$ such that
	\begin{equation*} 
	2-\e:=\max_{s\in [0,l_\gamma]^*} \mu_\gamma(\{s\}).
	\end{equation*}
	By Lemma \ref{lemma:measure} there exists $\delta>0$ such that for any $s_1<s_2$ with $s_2-s_1\le \delta$ it holds 
	\begin{equation}\label{E_curv_2}
	\mu_\gamma([s_1,s_2])<2-\frac{\e}{2}.
	\end{equation}
	It is an elementary fact that there exists $c_\e>0$ such that for any $s_1,s_2 \in [0,l_\gamma]^*$ for which \eqref{E_curv_2} holds there exists $\xi\in S^1$ such that
	\begin{equation*}
	\dot\gamma(s)\cdot \xi>c_\e
	\end{equation*}
	for $\L^1$-a.e. $s\in [s_1,s_2]$. This immediately implies that for any $s,s' \in[s_1,s_2]$ we have
	\begin{equation*}
	|\gamma(s)-\gamma(s')|\ge |(\gamma(s)-\gamma(s'))\cdot \xi|\ge c_\e|s-s'|.
	\end{equation*} 
	This proves that there exist $\delta>0$ and $c_\e>0$ depending only on $\gamma$ such that for every $s,s'\in [0,l_\gamma]^*$ with $|s-s'|\le \delta$ it holds 
	$|\gamma(s)-\gamma(s')|\ge c_\e|s-s'|$.
	Since $\gamma$ is injective and $[0,l_\gamma]^*$ is compact this uniform local Lipschitz invertibility implies global Lipschitz invertibility.
	\end{proof}

		\begin{lemma}\label{lemma:measure}
		Let $\mu$ be a finite, non-negative measure on $[0,1]$ and let $a := \max_{x \in [0,1]} \mu(\{x\})$. Then for every $b>a$ we can find $\delta=\delta(b)$ such that $\mu(I_\delta) < b$ for every subinterval $I_\delta \subset [0,1]$ of length $2\delta$. 
	\end{lemma}
	
	\begin{proof}
		The claim is trivial if $a=0$: indeed, in this case consider the function 
		\begin{equation*}
		f(x) := \int_0^x d\mu(y) = \mu([0,x])
		\end{equation*}
		which is uniformly continuous on $[0,1]$: therefore for any $b>0$ we can find $\delta>0$ such that 
		\begin{equation}\label{eq:unif_cont}
		\forall x,y \in [0,1]: \, x<y \text{ and } |x-y|< \delta \Rightarrow |f(x)-f(y)| = \mu([x,y])< b. 
		\end{equation}
		In case the measure has atoms, i.e. $a>0$, let us write $\mu = \mu^a + \tilde{\mu}$ where $\tilde{\mu}$ has no atoms and $\mu^a$ is the purely atomic measure given by 
		\begin{equation*}
		\mu^a := \sum_{n = 0}^\infty a_n \delta_{x_n}, \qquad (x_n)_{n \in \N} \subset [0,1], \qquad \sum_{n=0}^\infty a_n <+ \infty. 
		\end{equation*}
		Due to the convergence of the series defining $\mu^a$ (which is in turn a consequence of the finiteness of the measure $\mu$), we infer that there exists a finite set $F \subset [0,1]$ such that $\mu^a([0,1] \setminus F) < \frac{b-a}{2}$. 
		Denoting the elements of $F$ by $y_1 \le \ldots \le y_J$, we define $r := \min_{i=1, \ldots, J-1} |y_{i+1} - y_i|$. 
		Furthermore, applying the argument above to the atom-less measure $\tilde{\mu}$, we get from \eqref{eq:unif_cont} that there exists $\delta'=\delta'\left( \frac{b-a}{2}\right)$ such that $\tilde{\mu}(I_\delta')<\frac{b-a}{2}$ for every interval of length $2\delta'$.  
		If we now choose $\delta= \min \{\delta', r/2\}$, we can estimate
		\begin{equation*}
		\begin{split}
		\mu(I_{\delta}) & = \mu^a(I_{\delta}) + \tilde{\mu}(I_{\delta}) \\
		& \le \mu^a(I_{\delta} \cap F) + \mu^a(I_{\delta} \setminus F) + \frac{b-a}{2} \\
		& \le a +  \mu^a([0,1] \setminus F) + \frac{b-a}{2} \\
		& \le a + \frac{b-a}{2} + \frac{b-a}{2} = b
		\end{split}
		\end{equation*} 
		for every interval $I_{\delta}$ of length less than $\delta$. 
		Notice that in the third line we have used the fact that in $I_{\delta} \cap F$ there is at most one atom of $\mu$ (since $\delta\le r/2$). This completes the proof. 
		\end{proof}

	For future references we introduce the following class of simple curves:
	\begin{definition}
	We say that a bi-Lipschitz curve $\gamma:[0,l_\gamma]^*\to \gamma([0,l_\gamma]^*)\subset \R^2$ such that $|\gamma'|(s)=1$ for $\L^1$-a.e. $s\in [0,l_\gamma]^*$
	is \emph{$(c_S,M,L)$-admissible} if 
	\begin{equation*}
	\inf_{[0,l_\gamma]^*}|b\circ\gamma|\ge c_S, \qquad \TV \gamma' \le M \qquad \mbox{and} \qquad \Lip(\gamma^{-1})\le L.
	\end{equation*}
	\end{definition}

	We recall the well-known Jordan theorem on simple curve in $\R^2$.
	\begin{theorem}\label{T_Jordan}
	Let $C$ be a cycle in $\R^2$. Then $\R^2\setminus C$ has two connected components and one of them is bounded. 
	We denote by $\Int (C)$ the bounded connected component of $\R^2\setminus C$. 
	
	If moreover $C$ can be parametrized by a Lipschitz $\gamma:[0,l_\gamma]^*\to C$ with unit
	speed and such that $\TV_{[0,l_\gamma]^*}\gamma'< \infty$, then there exists $E\subset [0,l_\gamma]^*$ at most countable such that for every $s\in [0,l_\gamma]^*\setminus  E$
	and every $v\in \mathbb S^1$ such that $v\cdot (\gamma'(s))^\perp>0$ there exists $t>0$:
	\begin{equation*}
	\gamma(s)-\deg(\gamma)t'v\in \Int (C) \qquad \mbox{for every }t'\in (0,t),
	\end{equation*}
	where $\deg(\gamma)$ denotes the topological degree of the curve $\gamma$ with respect to a point $\bar x \in \Int(C)$.
	\end{theorem}
	In this work we will only consider the topological degree of simple curves in $\R^2$, also known as winding number: 
	roughly speaking it is equal to 1 if the curve turns around an interior point counterclockwise and $-1$ if it does it clockwise.
	
\section{Change of variables}\label{S_cov}
	In the following proposition we introduce a change of variables between two level sets and we summarize the properties that we need in Section \ref{S_Lip} in order to compare the
	evolution of two trajectories of the flow.
	\begin{proposition}\label{P_cov}
		Let $h_1,h_2 \in R$ with $h_2<h_1$ and $\gamma_1=\gamma^i_{h_1},\gamma_2=\gamma^{i'}_{h_2}$ be defined in Theorem \ref{T_Bourgain}. Assume that 
		$\gamma_1,\gamma_2$ are $(c_S,M,L)$-admissible curve and
		\begin{equation}\label{E_incl_int}
		\Int(C^{j_1}_{h_1})\subset \Int(C^{j_2}_{h_2}).
		\end{equation}
		We denote by
		\begin{equation*}
		A:= \Int(C^{j_2}_{h_2}) \setminus \overline{\Int(C^{j_1}_{h_1})}
		\end{equation*}
		the open region between the two curves and we assume that $\deg(\gamma_1)=-1= \deg(\gamma_2)$ and $h_1>h_2$.
		Then there exist $D_1\subset [0,l_{1}]^*$ and $D_2\subset [0,l_{2}]^*$ enjoying the following properties:
		\begin{enumerate}
			\item there exist $N\in \N$ and  two families of pairwise disjoint intervals $I_{1,j}:=[s_{1,j}^-,s_{1,j}^+]\subset D_1$ and  $I_{2,j}:=[s_{2,j}^-,s_{2,j}^+]\subset D_2$
				for $j=1,\ldots,N$ such that 
				\begin{equation*} 
				\bigcup_{j=1}^NI_{1,j}=D_1 \qquad \mbox{and} \qquad \bigcup_{j=1}^NI_{2,j}=D_2.
				\end{equation*}
			\item There exist two constants $\tilde c_1=\tilde c_1(c_S,M,L)>0$ and $\tilde c_2=\tilde c_2(c_S,M)>0$ such that
				\begin{equation}\label{E_est_h_1}
				l_{1}-|D_1|\le \tilde c_1|h_1-h_2| +\tilde c_2 |Db|(A) \qquad \mbox{and} \qquad 	l_2-|D_2|\le \tilde c_1|h_1-h_2| +\tilde c_2 |Db|(A).
				\end{equation}
			\item for every $j=1,\ldots,N$ there exists $e_j\in S^1$ such that for $\L^1$-a.e. $s\in I_{1,j}$ it holds 
				\begin{equation*}
				\dot\gamma_1(s)\cdot e_j \ge \frac{\sqrt 2}{2}|\dot\gamma_1(s)|
				\end{equation*}
				and for $\mathcal L^1$-a.e. $s\in I_{2,j}$ it holds 
				\begin{equation*}
				\dot\gamma_2(s)\cdot e_j \ge \frac{\sqrt 2}{2}|\dot\gamma_2(s)|.
				\end{equation*}
		\end{enumerate}
		In particular $\gamma_1\llcorner I_{1,j}$ and $\gamma_2\llcorner I_{2,j}$ are Lipschitz graphs in the same coordinate system: 
		more precisely there exist two bilipschitz functions $Y_{1,j}:I_{1,j}\to [0,l_j], \,Y_{2,j}:I_{2,j}\to [0,l_j]$ and two 1-Lipschitz functions $f_{1,j},f_{2,j}:[0,l_j]\to \R$ such that 
		\begin{equation}\label{E_Yf1}
		\gamma_1(s)=\gamma_1(s_{1,j}^-)+Y_{1,j}(s)e_j+f_{1,j}(Y_{1,j}(s))n_j \qquad \forall s\in I_{1,j}
		\end{equation}
		and 
		\begin{equation}\label{E_Yf2}
		\gamma_2(s)=\gamma_1(s_{1,j}^-)+Y_{2,j}(s)e_j+f_{2,j}(Y_{2,j}(s))n_j \qquad \forall s\in I_{2,j},
		\end{equation}
		where $n_j=e_j^\perp$. Moreover for $\L^1$-a.e. $s_1\in I_{1,j}$, $s_2\in I_{2,j}$ it holds
		\begin{equation}\label{E_Y_{1,j}}
		Y_{1,j}'(s_1)=\frac{b(\gamma_1(s_1))\cdot e_j}{|b(\gamma_1(s_1))|} \qquad \mbox{and} \qquad 
		Y_{2,j}'(s_2)=\frac{b(\gamma_2(s_2))\cdot e_j}{|b(\gamma_2(s_2))|}.
		\end{equation}
		
		Let $X_{1,2}:D_1\to D_2$ be the bi-Lipschitz change of variables defined by 
		\begin{equation*}
		X_{1,2}(s)=Y_{2,j}^{-1}(Y_{1,j}(s)) \qquad \forall s \in I_{1,j}
		\end{equation*}
		as $j$ goes from 1 to $N$.
		\begin{enumerate}
			\item[(4)] for any $j=1,\ldots,N$ and any $s\in I_{1,j}$ it holds
				\begin{equation*}
				0<d(s):=f_{2,j}(Y_{2,j}(X_{1,2}(s)))-f_{1,j}(Y_{i,j}(s))\le \frac{2\sqrt 2}{c_S}|h_1-h_2|,
				\end{equation*}
				i.e.
				\begin{equation}\label{E_max_dist}
				\sup_{D} |\gamma_1-\gamma_2\circ X_{1,2}|\le  \frac{2\sqrt 2}{c_S}|h_1-h_2|.
				\end{equation}
			\item[(5)]  for any $j=1,\ldots,N$ and any $s\in I_{1,j}$ it holds
				\begin{equation*}
				\gamma_1(s)+tn_j\in A \qquad \forall t \in (0, d(s)).
				\end{equation*}
			\item[(6)] for $j=1,\ldots,N$ the sets
				\begin{equation}\label{E_def_Ej}
				E_j:=\bigcup_{s\in I_{1,j}}\bigcup_{t\in [0,d(s)]}\{\gamma_1(s)+tn_j\}
				\end{equation}
				are pairwise disjoint.
			\item[(7)] The map $X_{1,2}$ is monotone: i.e. for every $s_1<s_1'$ in $D_1$ it holds 
			\begin{equation*}
			X_{1,2}([s_1,s_1']\cap D_1)= [X_{1,2}(s_1),X_{1,2}(s_1')]\cap D_2.
			\end{equation*}
		\end{enumerate}	
	\end{proposition}

	\begin{figure}
	\centering
	\def\svgwidth{0.9\columnwidth}
	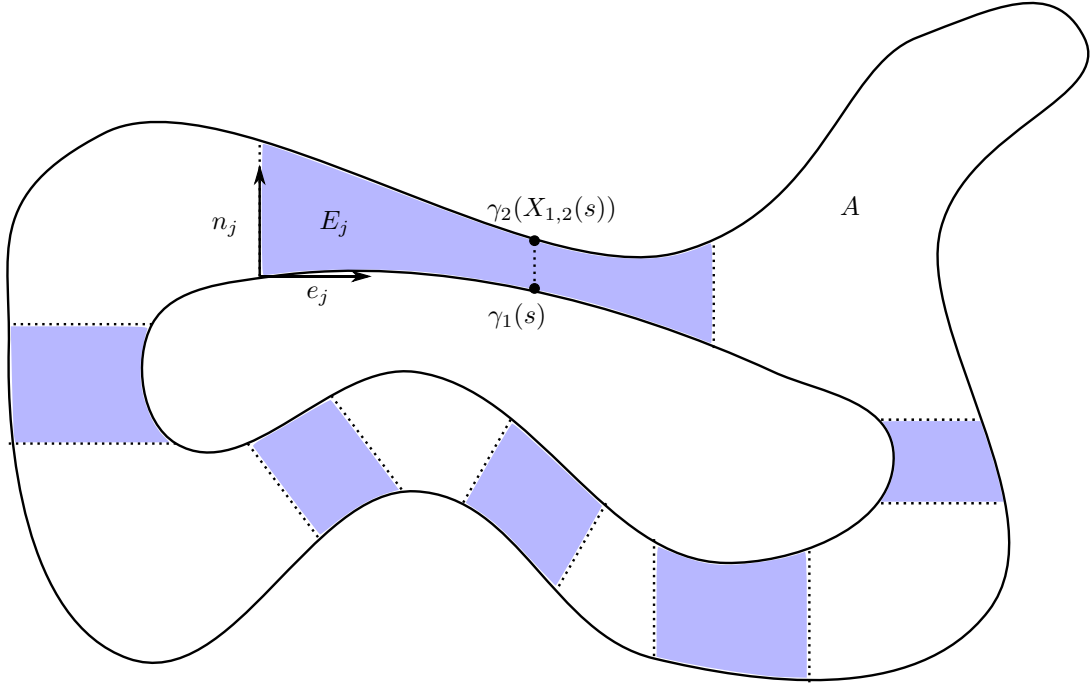
	\caption{Two cycles $\gamma_1$ and $\gamma_2$ as in Proposition \ref{P_cov}}\label{F_cov}
	\end{figure}

	The situation in Proposition \ref{P_cov} is represented in Figure \ref{F_cov}. In a few words the proposition above asserts that it is possible to map a part of the trajectory 
	on $\gamma_1$ onto $\gamma_2$. The fundamental property is that for every $j=1,\ldots,N$ the segments $\{ \gamma_1(s)+tn_j \}$ in \eqref{E_def_Ej} are parallel. 
	In that way the comparison of the evolutions of points on $\gamma_1$ and $\gamma_2$ is locally reduced to the straightforward case of a shear flow.
	The error due to the evolution of the points in $\gamma_1([0,l_1]^*\setminus D_1)$ and $\gamma_2([0,l_2]^*\setminus D_2)$ will be estimated by \eqref{E_est_h_1}.

	\begin{remark}\label{R_signs}
	It will be clear from the proof that the same statement holds true if instead of assuming
	\begin{equation*}
	h_1>h_2, \quad \Int(C^{j_1}_{h_1})\subset \Int(C^{j_2}_{h_2}), \quad \deg(\gamma_1)=\deg(\gamma_2)=-1, 
	\end{equation*}
	we assume one of the following:
	\begin{enumerate}
	\item $h_1>h_2, \quad  \Int(C^{j_2}_{h_2})\subset \Int(C^{j_1}_{h_1}), \quad \deg(\gamma_1)=\deg(\gamma_2)=1$;
	\item $h_1<h_2, \quad  \Int(C^{j_2}_{h_2})\subset \Int(C^{j_1}_{h_1}), \quad \deg(\gamma_1)=\deg(\gamma_2)=-1$; 
	\item $h_1<h_2, \quad  \Int(C^{j_1}_{h_1})\subset \Int(C^{j_2}_{h_2}), \quad \deg(\gamma_1)=\deg(\gamma_2)=1$.
	\end{enumerate}
	In cases (2) and (3) we need to consider $n_j:=-e_j^\perp$ instead of $n_j=e_j^\perp$.
	\end{remark}
	
	In order to prove Proposition \ref{P_cov} we need some lemmas.
		
	\begin{lemma}\label{L_interpolation}
	Let $\gamma:[0,l_\gamma]^*\to \R^2$ be a simple curve with finite turn and Lipschitz inverse. 
	Then there exist $K\in \N$, $l_K>0$, a bi-Lipschitz map $\tilde y_K:[0,l_\gamma]\to[0,l_K]$ and
	$0=s_1<\ldots<s_K=l_\gamma$ such that the following conditions are satisfied:
	\begin{enumerate}
	\item $\tilde y_K(0)=0$ and $\tilde y_K(l_{\gamma})=l_{K}$;
	\item $\gamma_K:=\gamma\circ \tilde y_K^{-1}: [0,l_K]\to \R^2$ is a simple curve with finite turn and Lipschitz inverse;
	\item for every $i=1,\ldots,K-1$ and for every $y\in (\tilde y_K(s_i),\tilde y_K(s_{i+1}))$ it holds 
	\begin{equation}\label{E_514}
	\dot\gamma_K(y)=e_i, \qquad \mbox{where}\qquad  e_i=\frac{\gamma(s_{i+1})-\gamma(s_i)}{|\gamma(s_{i+1})-\gamma(s_i)|};
	\end{equation}
	\item for every $i=1,\ldots,K-1$ and for $\L^1$-a.e. $s\in (s_i,s_{i+1})$ it holds
		\begin{equation}\label{E_1/2Lip}
		\dot\gamma(s)\cdot e_i \ge \frac{2}{\sqrt 5};
		\end{equation}
	\end{enumerate}
	\end{lemma}
	\begin{proof}
	First notice that any choice of $s_2<\ldots<s_{K-1}$ as in the statement uniquely selects $l_K>0$, the bi-Lipschitz map $\tilde y_K$ 
	and a continuous curve $\gamma_K:[0,l_K]\to \R^2$ such that $(1)$ and $(3)$ hold. Now we show that we can choose $s_2<\ldots<s_{K-1}$
	in such a way that also $(2)$ and $(4)$ are satisfied. In order to get $(4)$ we claim that if $\mu_\gamma((s_i,s_{i+1}))\le \frac{1}{\sqrt 5}$, then 
	\eqref{E_1/2Lip} holds. In fact from \eqref{E_514} it follows that
	\begin{equation*}
	\int_{s_i}^{s_{i+1}}\dot\gamma(s)\cdot n_i ds =0,
	\end{equation*}
	therefore
	\begin{equation*}
	\sup_{s\in (s_i,s_{i+1})}|\dot\gamma(s)\cdot n_i|\le \TV(\dot\gamma\cdot n_i \llcorner (s_i,s_{i+1})) \le \mu_\gamma((s_i,s_{i+1}))\le \frac{1}{\sqrt 5}.
	\end{equation*}
	Since   $|\dot\gamma(s)\cdot n_i|\le \frac{1}{\sqrt 5} \Rightarrow |\dot\gamma(s)\cdot e_i|\ge \frac{2}{\sqrt 5}$, then it remains to check that for $\L^1$-a.e. $s\in (s_i,s_{i+1})$ it holds
	$\dot\gamma(s)\cdot e_i\ge 0$. This easily follows from $\TV(\dot\gamma \cdot e_i\llcorner (s_i,s_{i+1}))\le \mu_\gamma((s_i,s_{i+1}))$ and concludes the proof of the claim.
	
	Thanks to Lemma \ref{L_Lip_inverse}, in order to prove (2) it is sufficient to check that $\gamma_K$ is one to one.
	If there exists $s<s'$ such that $\gamma_K(\tilde y_K(s))=\gamma_K(\tilde y_K(s'))$ then $\mu_{\gamma_K}((\tilde y_K(s),\tilde y_K(s')))\ge 2$. 
	Moreover it is easy to check that for every $1\le i_1 < i_2\le K$ it holds $\mu_{\gamma_K}((\tilde y_K(s_{i_1}),\tilde y_K(s_{i_2})))\le \mu_{\gamma}((s_{i_1},s_{i_2}))$ therefore,
	letting $i$ and $i'$ be such that $s\in (s_i,s_{i+1}), s'\in (s_{i'},s_{i'+1})$, it holds $\mu_\gamma((s_{i}, s_{i'+1}))\ge 2$.
	Since by assumption $\mu_\gamma$ has no atoms of size 2, by Lemma \ref{lemma:measure} it follows that there exists $\delta >0$ such that $s_{i'+1}-s_i\ge \delta$.
	Therefore 
	\begin{equation}\label{E_lower}
	|\gamma(s_{i'+1})-\gamma(s_i)|\ge L\delta,
	\end{equation}
	where $L$ is the Lipschitz constant of $\gamma^{-1}$. 
	Choosing $s_i$ such that $\max_{i}s_{i+1}-s_i <L\delta/4$, then 
	\begin{equation*}
	\max_s |\gamma(s)-\gamma_K(\tilde Y_K(s))|\le L\delta/4
	\end{equation*}
	so that
	\begin{equation*}
	\begin{split}
	|\gamma(s_{i'+1})-\gamma(s_i)|\le & |\gamma(s_{i'+1}) - \gamma_K(\tilde y_K(s_{i'+1}))|+  |\gamma_K(\tilde y_K(s_{i'+1}))-\gamma_K(\tilde y_K(s'))|
	+ |\gamma_K(\tilde y_K(s'))-\gamma_K(\tilde y_K(s))| \\& +|\gamma_K(\tilde y_K(s))-\gamma_K(\tilde y_K(s_i))| + |\gamma_K(\tilde y_K(s_i))-\gamma(s_i)| \\
	< & L\delta/4 +  L\delta/4  + 0 +  L\delta/4 + L\delta/4 \\
	=& L\delta
	\end{split}
	\end{equation*}
	and this is in contradiction with \eqref{E_lower}. 
	Being $\mu_\gamma$ a finite measure it is possible to choose $s_2<\ldots< s_{K-1}$ so that $\mu_\gamma((s_i,s_{i+1}))\le \frac{1}{\sqrt 5}$ for every $i=1,\ldots, K-1$ and 
	so that  $\max_{i}s_{i+1}-s_i <L\delta/4$. With this choice also properties (4) and (2) are granted and this concludes the proof of the lemma.
	\end{proof}

	\begin{lemma}\label{L_overlapping}
	Let $a>0$, $\gamma:[0,l_\gamma]^*\to \R^2$ be a simple curve with finite turn such that $\gamma^{-1}$ is $L$-Lipschitz. 
	Denote by 
	\begin{equation*}
	\begin{split}
	\mathcal B_a(\gamma):=\{s&\in [0,l_\gamma]: \gamma \mbox{ is differentiable at }s \\
	&\mbox{ and } \exists s'\in[0,l_\gamma],t,t'\in [-a,a] \mbox{ such that } s'\neq s \mbox{ and }\gamma(s)+tn(s)=\gamma(s')+t'n(s')\},
	\end{split}
	\end{equation*}
	where $n(s)=\dot\gamma(s)^\perp$.
	Then there exists an absolute constant $\underline c_0>0$ such that 
	\begin{equation*}
	\H^1(\mathcal B_a(\gamma))\le \underline c_0(1+L)a\TV_{[0,l_\gamma]^*}\gamma'.
	\end{equation*}
	\end{lemma}
	\begin{proof}
	We claim that if there exist $t,t'\in[-a,a]$ such that $\gamma(s)+tn(s)=\gamma(s')+t'n(s')$, then 
	\begin{equation*}
	\mu_\gamma([s,s'])\ge\frac{|s'-s|}{(1+2L) a}.
	\end{equation*}
	In particular
	\begin{equation*}
	s\in \mathcal B_a(\gamma) \Rightarrow M_{\mu_\gamma}(s)\ge \frac{1}{(1+2L)a},
	\end{equation*}
	where $M_{\mu_\gamma}$ denotes the maximal function of the measure $\mu_\gamma$.
	The lemma follows from this claim by the weak $L^1$ continuity property of the maximal operator.
	Let us prove the claim:
	since $\gamma(s)+tn(s)=\gamma(s')+t'n(s')$ with $t,t'\in[-a,a]$ and $\gamma^{-1}$ is $L$-Lipschitz, then $|s'-s|\le 2La$.
	From $ \gamma(s)+tn(s)=\gamma(s')+t'n(s') $ it also follows that 
	\begin{equation*}
	\gamma(s)+tn(s)-\gamma(s') \parallel n(s') \qquad \mbox{i.e.} \qquad 
	(\gamma(s)+tn(s)-\gamma(s'))\cdot \gamma'(s')=0.
	\end{equation*}
	Therefore
	\begin{equation*}
	\begin{split}
	(\gamma'(s)-\gamma'(s'))\cdot (\gamma(s)+tn(s)-\gamma(s')) = & \gamma'(s)\cdot (\gamma(s)+tn(s)-\gamma(s')) \\
	=& \gamma'(s)\cdot(\gamma(s)-\gamma'(s)) \\
	=& \gamma'(s)\cdot (\gamma(s)+\gamma'(s)(s'-s)-\gamma(s')) - |\gamma'(s)|^2(s'-s),
	\end{split}
	\end{equation*}
	so in particular
	\begin{equation*}
	\begin{split}
	|s'-s|\le & |(\gamma'(s)-\gamma'(s'))\cdot (\gamma(s)+tn(s)-\gamma(s'))| + |\gamma(s)+\gamma'(s)(s'-s)-\gamma(s')| \\
	\le & \mu_\gamma([s,s'])a + \mu_\gamma([s,s'])|s'-s|.
	\end{split}
	\end{equation*}
	Therefore, since $|s'-s|\le 2La$,
	\begin{equation*}
	M_{\mu_\gamma}(s)\ge \frac{\mu_{\gamma}([s,s'])}{|s'-s|} \ge \frac{1}{|s'-s|+a}\ge \frac{1}{(1+2L)a}.
	\end{equation*}
	This concludes the proof of the claim and therefore of the lemma.
	\end{proof}

	A simple application of the slicing theory of $\BV$ functions provides the following lemma. See \cite{AFP_book}.
	
	\begin{lemma}\label{L_Dg}
	Let $f_1,f_2:D\subset \R\to \R$ two measurable functions with $f_1\le f_2$ and let $g\in \BV(\R^2)$.
	Denote by
	\begin{equation*}
	E:=\{(x_1,x_2)\in D\times \R: f_1(x_1)\le x_2\le f_2(x_1)\}.
	\end{equation*}
	Then 
	\begin{equation*}\label{E_Dg}
	\int_D|g^*(f_2(x_1))-g^*(f_1(x_1))|dx_1 \le |Dg|(E),
	\end{equation*}
	where $g^*$ denotes the precise representative of $g$.	
	\end{lemma}
		
	\begin{lemma}\label{L_folding}
	Let $h\in R$ and let $\gamma:=\gamma^j_h$ parametrizing a cycle $C^j_h$ as in Theorem \ref{T_Bourgain}. Assume moreover that $\gamma$ is $(c_S,M,L)$-admissible. 
	Then there exists $c=c(c_S,M)>0$ such that
	\begin{equation*}
	|Db|(\Int (C^j_h))\ge c l_\gamma.
	\end{equation*}
	\end{lemma}
	\begin{proof}
	We can assume without loss of generality that $M\ge 1$.
	Since $\TV \gamma'\le M$, then there exist an absolute constant $c_1>0$ and $[a,b]\subset [0,l_\gamma]^*$ such that
	\begin{equation*}
	\TV_{(a,b)}\gamma' \le 1 \qquad \mbox{and} \qquad b-a\ge \frac{c_1}{M}l_\gamma.
	\end{equation*}
	In particular there exists $e\in \mathbb{S}^1$ such that for $\L^1$-a.e. $s\in (a,b)$ it holds
	\begin{equation*}
	\gamma'(s)\cdot e \ge \frac{\sqrt 2}{2}
	\end{equation*}
	so that
	\begin{equation*}
	(\gamma(b)-\gamma(a))\cdot e \ge \frac{\sqrt{2}c_1}{2M}l_\gamma.
	\end{equation*}
	Let $n=e^\perp$ if $\deg(\gamma)=1$ and $n=-e^\perp$ if  $\deg(\gamma)=-1$ so that by Theorem \ref{T_Jordan} for every $s\in (a,b)$
	there exists $t>0$ such that for every $t'\in(0,t)$ it holds $ \gamma(s)+t'n\in \Int(C^j_h)$. We denote by
	$g(s)\in [0,l_\gamma]^*$ the unique point for which
	\begin{equation*}
	\gamma(g(s))=\gamma(s)+ \bar t n, \qquad \mbox{where}\qquad \bar t:=\sup\{t>0: \gamma(s)+t'n\in \Int(C^j_h)\,  \forall t'\in (0,t)\}.
	\end{equation*} 
	By an elementary topological consideration we have that for $\L^1$-a.e. $s \in (a,b)$ it holds $\gamma'(g(s))\cdot e \le 0$.
	Since $|b\circ \gamma|\ge c_S$ it follows that for $\L^1$-a.e. $s \in (a,b)$ it holds
	\begin{equation*}
	|b(\gamma(s))-b(\gamma(g(s))|\ge \frac{\sqrt{2}}{2}c_S.
	\end{equation*}
	By Lemma \ref{L_Dg} this implies that
	\begin{equation*}
	|Db|(\Int(C^j_h))\ge \int_a^b |b(\gamma(s)-b(\gamma(g(s))|ds \ge \frac{1}{2}c_S(b-a)\ge \frac{c_1c_S}{2M}l_\gamma. \qedhere
	\end{equation*}
	\end{proof}

	\begin{proof}[Proof of Proposition \ref{P_cov}]
	We first construct the change of variable and then we check all the properties. 
	
	\emph{Step 1}. Construction of $\tilde X_{1,2}$.
	Consider the piecewise affine interpolating curve $\gamma_K$ obtained by Lemma \ref{L_interpolation} with $\gamma=\gamma_1$.
	Then the following properties hold:
	\begin{enumerate}
		\item	$\tilde y_K(0)=0$ and $\tilde y_K(l_{h_1})=l_{\gamma_K}$;
		\item for every $i=1,\ldots, K$ it holds $\gamma_1(s_i)=\gamma_K(\tilde y_k(s_i))$;
		\item for every $i=1,\ldots,K-1$ and for every $y\in (\tilde y_K(s_i),\tilde y_K(s_{i+1}))$ it holds 
		\begin{equation*}
		\dot\gamma_K(y)=e_i, \qquad \mbox{where}\qquad  e_i=\frac{\gamma_1(s_{i+1})-\gamma_1(s_i)}{|\gamma_1(s_{i+1})-\gamma_1(s_i)|};
		\end{equation*}
		\item  for every $i=1,\ldots,K-1$ and for $\L^1$-a.e. $s\in (s_i,s_{i+1})$ it holds
		\begin{equation}\label{E_cond_4}
		\dot\gamma_1(s)\cdot e_i \ge \frac{2}{\sqrt 5}|\dot\gamma_1(s)|.
		\end{equation}
	\end{enumerate}
	Up to refining the choice of the mesh points $s_i$ we can also assume that
	\begin{enumerate}
		\item[(5)] $\max\{s_{i+1}-s_i: i=1,\ldots,K-1\}\le 2 |h_1-h_2|$.
	\end{enumerate}
	Denote by $F:=[0,l_{\gamma_K}]^*\setminus \mathcal B_{\bar a}(\gamma_K)$ where $\mathcal B_{\bar a}(\gamma_K)$ is obtained by Lemma \ref{L_overlapping} with 
	$\bar a=(\frac{\sqrt 5}{c_S}+1)|h_1-h_2|$ and set $\tilde D:=(\tilde y_K)^{-1}(F)$.
	By Lemma \ref{L_overlapping} it holds
	\begin{equation*}
	l_{\gamma_K}-\L^1(F)\le \underline c_0(1+L)\bar a \TV \gamma_K' \le \underline c_0(1+L)\bar a\TV \gamma_1' \le \underline c_0(1+L)\bar a M.
	\end{equation*}
	Thanks to Property (4) above we deduce that
	\begin{equation}\label{E_est_tildeD}
	l_{\gamma_1}-\L^1(\tilde D)\le \frac{\sqrt 5}{2} \underline c_0(1+L)\bar aM.
	\end{equation} 
	Let $\tilde D'\subset \tilde D$ be defined by
	\begin{equation*}\label{E_D2}
	\begin{split}
	\tilde D':=\bigg\{s'\in \tilde D:& \exists i=1,\ldots,K, \exists s''\in [0,l_{h_2}]^*, \exists t\in\left[0, \frac{\sqrt 5}{c_S}|h_1-h_2|\right]: \\
	& s'\in (s_i,s_{i+1}), \gamma_1(s')+tn_i=\gamma_2(s''), \quad \mbox{and} \quad \forall t'\in (0,t), \, \gamma_1(s')+t'n_i \in A \bigg\}.
	\end{split}
	\end{equation*}
	We now estimate $|\tilde D\setminus \tilde D'|$. We distinguish two cases:
	\begin{equation*}
	\tilde D\setminus \tilde D'= B_1 \cup B_2,
	\end{equation*}
	where
	\begin{equation*}
	B_1:=\left\{s'\in \tilde D: \exists i=1,\ldots,K: s'\in (s_i,s_{i+1}) \quad \mbox{and} \quad \forall t \in \left[0, \frac{\sqrt 5}{c_S}|h_1-h_2|\right], \, H(\gamma_1(s')+tn_i)>h_2\right\}
	\end{equation*}
	and
	\begin{equation*}
	\begin{split}
	B_2:=\bigg\{s'\in \tilde D\setminus \tilde D': & \exists i=1,\ldots,K, \exists t \in \left[0, \frac{\sqrt 5}{c_S}|h_1-h_2|\right]: \\
	& s'\in (s_i,s_{i+1}), H(\gamma_1(s')+tn_i)=h_2\quad \mbox{and} \quad \forall t'\in (0,t), \, \gamma_1(s')+t'n_i\in A\bigg\}.
	\end{split}
	\end{equation*}
	We first estimate the measure of $B_2$. By definition there exists $\gamma_1(s')+tn_i\in C^{j'}_{h_2}$ with $j'\ne j_2$. Therefore $C^{j'}_{h_2}\subset A$.
	Since $C^{j'}_{h_2}\cap C^{j_2}_{h_2}=\emptyset = C^{j'}_{h_2}\cap C^{j_1}_{h_1}$ it follows that $\Int( C^{j'}_{h_2})\subset A$.
	By Lemma \ref{L_folding} we have that there exists $c=c(c_s,M)>0$ such that
	\begin{equation*}
	c\H^1(C^{j'}_{h_2}) \le |Db|(\Int( C^{j'}_{h_2}))\le |Db|(A).
	\end{equation*}
	Moreover by \eqref{E_cond_4} it follows that 
	\begin{equation}\label{E_est_B2}
	\L^1(B_2) \le \frac{\sqrt{5}}{2}\sum_{j'\in J'} \H^1(C^{j'}_{h_2}) \le \frac{\sqrt 5}{2c}|Db|(A),
	\end{equation}
	where $J'$ denotes the set of indexes $j'\ne j_2$ such that $C^{j'}_{h_2}\subset A$.
	
	Now we estimate $\L^1(B_1)$: for every $i=1,\ldots, K$ let $\tilde  f_{1,i}:\tilde y_K([s_i,s_{i+1}])\to \R$ be defined by the following relation:
	\begin{equation*}
	\tilde f_{1,i}(y)=(\gamma_1((\tilde y_K)^{-1}(y))-\gamma_1(s_i))\cdot n_i.
	\end{equation*}
	Let $s\in B_1\cap (s_i,s_{i+1})$. 
	Then for every $t\in (0,\frac{\sqrt 5}{c_S}|h_1-h_2|])$ it holds $H(\gamma_1(s)+tn_i)>h_2$.
	In particular there exists $\tilde t \in (0,\frac{\sqrt 5}{c_S}|h_1-h_2|])$ such that $\nabla H(\gamma_1(s)+\tilde tn_i)\cdot n_i \ge -\frac{c_S}{\sqrt 5}$. Since $\nabla H(\gamma_1(s))\cdot n_i\le -\frac{2c_S}{\sqrt 5}$ we have
	\begin{equation*}
	\mathrm{Osc}_{t\in [0,\frac{\sqrt 5}{c_S}|h_1-h_2|]} \nabla H(\gamma_1(s)+ tn_i)\cdot n_i \ge \frac{c_S}{\sqrt 5}.
	\end{equation*}
	Applying Lemma \ref{L_Dg} with $D=B_1\cap (s_i,s_{i+1})$, $f_1= \tilde f_{1,i}$ and $f_2=\tilde f_{1,i}+\frac{\sqrt 5}{c_S}|h_1-h_2|$ and summing the contributions for
	$i=1,\ldots,K$ we get that
	\begin{equation}\label{E_est_B1}
	\L^1(B_1)\le \frac{\sqrt 5}{c_S}|Db|(A).
	\end{equation}
	By \eqref{E_est_B2} and \eqref{E_est_B1} we obtain
	\begin{equation}\label{E_est_tildeD'}
	\L^1(\tilde D \setminus \tilde D') \le \left(\frac{\sqrt 5}{2c}+\frac{\sqrt 5}{c_S} \right)|Db|(A).
	\end{equation}
	
	For every $i=1,\ldots, K$ and every $s'\in \tilde D'\cap (s_i,s_{i+1})$ let $\tilde X_{1,2}(s')\in [0,l_{h_2}]^*$ and $d(s')>0$ be the unique values for which
	\begin{equation*}
	\gamma_2(\tilde X_{1,2}(s'))=\gamma_1(s')+d(s')n_i \qquad \mbox{and} \qquad H(\gamma_1(s_1)+d'n_i)\in (h_1,h_2) \quad \forall d'\in (0,d(s')).
	\end{equation*}
	Now we are in position to define 	and $\tilde f_{2,i}:\tilde y_K([s_i,s_{i+1}])\cap \tilde D'\to \R$ by 
	\begin{equation*}
	\tilde f_{2,i}(y)=(\gamma_2(\tilde X_{1,2}((\tilde y_K)^{-1}(y)))-\gamma_1(s_i))\cdot n_i= \tilde f_{1,i}(y) + d((\tilde y_K)^{-1}(y))).
	\end{equation*}
	Notice that from Property (4) above we have that $\tilde f_{1,i}$ is 1/2-Lipschitz for every $i=1,\ldots,K$.
	Therefore, thanks to Property (5), it holds $\tilde f_{1,i}\le |h_1-h_2|$.
		
	Moreover notice that for $i=1,\ldots,K$ the sets
	\begin{equation*}
	\tilde E_i:= \bigcup_{s\in \tilde D' \cap (s_i,s_{i+1})}\bigcup_{d'\in [0,d(s)]}\{\gamma_1(s)+d'n_i\} \subset 
	\bigcup_{s\in \tilde D'\cap (s_i,s_{i+1})}\bigcup_{t\in [-a,a]} \{\gamma_K(\tilde y_K(s))+tn_i\}.
	\end{equation*}
	By Lemma \ref{L_overlapping} the sets 
	\begin{equation*}
	\left\{ \bigcup_{t\in[-a,a]}\{\gamma_K(\tilde y_K(s))+tn_i\}\right\}_{s\in \tilde D'}
	\end{equation*}
	are pairwise disjoint, therefore also $\{\tilde E_i\}_{i=1}^K$ are pairwise disjoint.
	
	By construction we have that $\tilde X_{1,2}$ enjoys the properties (3), (4), (5) and (6) of the statement: 
	the monotonicity of $\tilde X_{1,2}$ in the sense of the statement follows by an elementary topological argument since $\gamma_1$ and $\gamma_2$ 
	have the same degree and the sets $\{\tilde E_i\}_{i=1}^K$ are pairwise disjoint subsets of $A$.
		
	\emph{Step 2}. We define the function $X_{1,2}$ as an appropriate restriction of $\tilde X_{1,2}$ so that properties (1) and (2) are satisfied. Properties from (3) to (7) are
	obviously inherited by $\tilde X_{1,2}$.  
		
	Fix $i=1,\ldots, K-1$ and let
	\begin{equation*}
	\tilde D_3^i:= \{s\in \tilde D'\cap (s_i,s_{i+1}): \dot \gamma_2(\tilde X_{1,2}(s))\cdot e_i \ge \frac{\sqrt 2}{2}| \dot \gamma_2(\tilde X_{1,2}(s))|\}.
	\end{equation*}
	We claim that for every $i=1,\ldots, K$ there are finitely many connected components of $\tilde X_{1,2}(\tilde D_3^i)$ which contain a point $s$ for which 
	\begin{equation*}
	\dot\gamma_2(s)\cdot e_i \ge \frac{4}{3\sqrt 3}|\dot\gamma_2(s)|.
	\end{equation*}
	Let $i=1,\ldots, K$ and consider two different connected components $I$, $I'$ as above of $\tilde X_{1,2}(\tilde D_3^i)$.
	Let $s\in I$ and $s'\in I'$ be such that
	\begin{equation*}
	\dot\gamma_2(s)\cdot e_i \ge \frac{4}{3\sqrt 3}|\dot\gamma_2(s)| \qquad \mbox{and} \qquad \dot\gamma_2(s')\cdot e_i \ge \frac{4}{3\sqrt 3}|\dot\gamma_2(s')|.
	\end{equation*}
	Assume without loss of generality that $s<s'$. Since they belong to two different connected components of $\tilde X_{1,2}(\tilde D_3^i)$, there exists $\tilde s\in (s,s')$
	such that 
	\begin{equation*}
	\dot\gamma_2(s)\cdot e_i \le \frac{\sqrt 2}{2} |\dot\gamma_2(s)|.
	\end{equation*}
	In particular $\mu_\gamma([s,s'])\ge 2 \left(\frac{4}{3\sqrt 3}-\frac{\sqrt 2}{2} \right)c_S$. Since $\gamma_2$ has finite turn there can be at most finitely many of such connected
	components.
	We denote by $(I^i_{2,j})_{j=1}^{N_i}$ the connected components of $\tilde X_{1,2}(\tilde D_3^i)$ and we denote by $D_1^i$ their union.
	Up to an arbitrarily small restriction we can assume that $I^i_{2,j}=[s^{i,-}_{2,j},s^{i,-}_{2,j}]$ is closed.
	For every $j=1,\ldots, N_i$ we denote by $I^i_{1,j}=\tilde X_{1,2}^{-1}(I^i_{2,j})=[s^{i,-}_{1,j},s^{i,-}_{1,j}]$.
	Accordingly we define $Y^i_{1,j}:I^i_{1,j}\to [0,l^i_j]$, $Y^i_{2,j}:I^i_{2,j}\to [0,l^i_j]$ and $f^i_{1,j},f^i_{2,j}:[0,l^i_j]\to \R$ so that
	\begin{equation}\label{E_Yf1}
		\gamma_1(s)=\gamma_1(s_{1,j}^{i,-})+Y^i_{1,j}(s)e_i+f^i_{1,j}(Y^i_{1,j}(s))e_i^\perp \qquad \forall s\in I^i_{1,j}
	\end{equation}
		and 
	\begin{equation}\label{E_Yf2}
		\gamma_2(s)=\gamma_1(s_{1,j}^{i,-})+Y^i_{2,j}(s)e_j+f_{2,j}(Y^i_{2,j}(s))e_i^\perp\qquad \forall s\in I^i_{2,j}.
	\end{equation}
	We denote by $D_1=\bigcup_{i=1}^KD_1^i$, $N=\sum_{i=1}^KN_i$ and we define $X_{1,2}$ as the restriction of $\tilde X_{1,2}$ to $D_1$. 
	Accordingly we define $I_{1,j},I_{2,j},Y_{1,j},Y_{2,j}$ as $j=1,\ldots,N$.

	Notice that \eqref{E_Y_{1,j}} follows immediately by taking the scalar product of \eqref{E_Yf1} and \eqref{E_Yf2} with $e_i$ and computing the derivative with respect to $s$.

	Now we check that the first inequality in \eqref{E_est_h_1} is satisfied: since for every $s\in \tilde D'\setminus D_1$ it holds
	\begin{equation*}
	\dot\gamma_1(s)\cdot e_i \ge \frac{2}{\sqrt 5}|\dot\gamma_1(s)|, \qquad 
	\dot\gamma_2(\tilde X_{1,2}(s))\cdot e_i \le \frac{4}{3\sqrt 3}|\dot\gamma_1(\tilde X_{1,2}(s))|
	\end{equation*}
	and $ |\dot\gamma_1(s)|,|\dot\gamma_2(\tilde X_{1,2}(s))| \ge c_S$, then there exists an absolute constant $\underline c_3>0$ such that
	\begin{equation*}
	|\dot\gamma_2(\tilde X_{1,2}(s))-\dot\gamma_1(s)|\ge \underline c_3 c_S.
	\end{equation*}
	Integrating on $\tilde D'\setminus D_1$ we get that
	\begin{equation}\label{E_est_D}
	|\tilde D'\setminus D|\le \frac{\sqrt 5}{2\underline c_3 c_S}|Db|(A). 
	\end{equation}
	From \eqref{E_est_tildeD}, \eqref{E_est_tildeD'} and \eqref{E_est_D} it follows that there exist $\tilde c_3=\tilde c_3(c_S,M,L)>0$ and $\tilde c_4=\tilde c_4(c_S,M)>0$ such that
	\begin{equation}\label{E_estD1}
	\L^1([0,l_1]^*\setminus D_1) \le \tilde c_3 |h_1-h_2| + \tilde c_4 |Db|(A),
	\end{equation}
	and this proves the first inequality in \eqref{E_est_h_1}.
	
	It remains to prove the second inequality in \eqref{E_est_h_1}.
	First we compare the lengths of $\gamma_1\llcorner I_j$ and of $\gamma_2\llcorner X_{1,2}(I_j)$.
	\begin{equation*}
	\begin{split}
	l(\gamma_1\llcorner I_j)=s_{1,j}^+-s_{1,j}^-=&\int_0^{l_j}\sqrt{1+|f_{1,j}'|^2(y)}dy \\
	l(\gamma_2\llcorner X_{1,2}(I_j))=&\int_0^{l_j}\sqrt{1+|f_{2,j}'|^2(y)}dy
	\end{split}
	\end{equation*}
	so that 
	\begin{equation*}
	\begin{split}
	|l(\gamma_2\llcorner X_{1,2}(I_j))-l(\gamma_1\llcorner I_j)|\le &~ \int_0^{l_j}||f_{2,j}'|-|f_{1,j}'||dy \\
	\le &~ \frac{\underline c_4}{c_S}\int_0^{l_j}|b(\gamma_2(Y_{2,j}^{-1}(y)))-b(\gamma_1(Y_{1,j}^{-1}(y)))|dy \\
	\le &~ \frac{\underline c_4}{c_S} |Db|(E_j)
	\end{split}
	\end{equation*}
	for an absolute constant $\underline c_4>0$,
	where we used in the first line that $v\to \sqrt{1+v^2}$ is 1-Lipschitz, in the second line that $\gamma_1,\gamma_2$ are $(c_S,M,L)$-admissible and finally Lemma \ref{L_Dg}.
	Since the sets $(E_j)_{j=1}^N$ are pairwise disjoint, summing on $j=1,\ldots,N$ we get
	\begin{equation}\label{E_DeltaLj}
	\sum_{j=1}^N|l(\gamma_2\llcorner X_{1,2}(I_j))-l(\gamma_1\llcorner I_j)|\le \frac{\underline c_4}{c_S}|Db|(A).
	\end{equation} 
	From \eqref{E_estD1} and \eqref{E_DeltaLj}, we have
	\begin{equation}\label{E_369}
	\begin{split}
	l(\gamma_2) \ge &\sum_{j=1}^N l(\gamma_2\llcorner I_j)\\
	\ge &\sum_{j=1}^Nl(\gamma_1\llcorner I_j)-\frac{\underline c_4}{c_S}|Db|(A)\\
	\ge &~ l(\gamma_1)- \L^1([0,l_{h_1}]\setminus D)-\frac{\underline c_4}{c_S}|Db|(A) \\
	\ge &~ l(\gamma_1)-  \tilde c_3 |h_1-h_2| -\tilde c_4|Db|(A)-\frac{\underline c_4}{c_S}|Db|(A) .
	\end{split}
	\end{equation}
	We now observe that exactly the same argument of this proof up to now can be repeated exchanging the role of $\gamma_1$ and $\gamma_2$. 
	Therefore from \eqref{E_369} and the analogous inequality exchanging $\gamma_1$ and $\gamma_2$ we get
	\begin{equation}\label{E_DeltaL}
	|l(\gamma_2)-l(\gamma_1)|\le \tilde c_3|h_1-h_2| +\left(\tilde c_4+\frac{\underline c_4}{c_S}\right)|Db|(A).
	\end{equation}
	From \eqref{E_DeltaLj} and \eqref{E_DeltaL} we get
	\begin{equation*}
	\begin{split}
	l(\gamma_2)-\sum_{j=1}^Nl(\gamma_2\llcorner I_{2,j}) \le &~ l(\gamma_2)-\sum_{j=1}^Nl(\gamma_1\llcorner I_{1,j})+\frac{\underline c_4}{c_S}|Db|(A) \\
	\le &~ l(\gamma_2)-l(\gamma_1)+ \L^1([0,l_{h_1}]\setminus D_1)+\frac{\underline c_4}{c_S}|Db|(A) \\
	\le &~ 2 \tilde c_3|h_1-h_2| +2\left(\tilde c_4+\frac{\underline c_4}{c_S}\right)|Db|(A).
	\end{split}
	\end{equation*}
	Setting $\tilde c_1=2\tilde c_3$ and $\tilde c_2=2\left(\tilde c_4+\frac{\underline c_4}{c_S}\right)$, this proves \eqref{E_est_h_1} and concludes the proof.
	\end{proof}

\section{Lusin-Lipschitz regularity of the flow}\label{S_Lip}
	We first  estimate the difference of the flow starting from two points belonging to the same trajectory.
	\begin{lemma}\label{L_same}
	Let $\gamma$ be a $(c_S,M,L)$-admissible curve. Then for every $t\ge 0$ and every $s_1,s_2\in [0,l_\gamma]^*$ it holds
	\begin{equation}\label{E_same}
	|X(t,\gamma(s_2))-X(t,\gamma(s_1))|\le \frac{\|b\|_{L^\infty}^2}{c_S}L|\gamma(s_2)-\gamma(s_1)|.
	\end{equation}
	\end{lemma}
	\begin{proof}
	Since $\partial_tX(t,\gamma(s))=b(X(t,\gamma(s)))$, then
	\begin{equation*}
	\frac{d}{dt} \int_{\gamma^{-1}(X(t,\gamma(s_1)))}^{\gamma^{-1}(X(t,\gamma(s_2)))}\frac{1}{|b(\gamma(s))|}ds=0.
	\end{equation*}
	Moreover for $\L^1$-a.e. $s \in [0,l_\gamma]^*$ it holds
	\begin{equation*}
	c_S\le |b(\gamma(s))|\le \|b\|_{L^\infty}.
	\end{equation*}
	Therefore
	\begin{equation*}
	\begin{split}
	\max_{t\ge 0} |\gamma^{-1}(X(t,\gamma(s_2)))-\gamma^{-1}(X(t,\gamma(s_1)))|\le &~ 
	\|b\|_{L^\infty}\int_{\gamma^{-1}(X(t_{\max},\gamma(s_1)))}^{\gamma^{-1}(X(t_{\max},\gamma(s_2)))} \frac{1}{|b(\gamma(s))|}ds\\
	= &~ \|b\|_{L^\infty}\int_{\gamma^{-1}(X(t_{\min},\gamma(s_1)))}^{\gamma^{-1}(X(t_{\min},\gamma(s_2)))} \frac{1}{|b(\gamma(s))|}ds \\
	\le &~ \frac{\|b\|_{L^\infty}}{c_S}\min_{t\ge 0} |\gamma^{-1}(X(t,\gamma(s_2)))-\gamma^{-1}(X(t,\gamma(s_1)))|.
	\end{split}
	\end{equation*}
	Since $\gamma^{-1}$ is $L$-Lipschitz, then for every $t\ge 0$ it holds
	\begin{equation*}
	|\gamma^{-1}(X(t,\gamma(s_2)))-\gamma^{-1}(X(t,\gamma(s_1)))|\le \frac{\|b\|_{L^\infty}}{c_S}|s_2-s_1|\le  \frac{\|b\|_{L^\infty}}{c_S}L|\gamma(s_2)-\gamma(s_1)|, 
	\end{equation*}
	which immediately implies \eqref{E_same} since $\gamma$ is $\|b\|_{L^\infty}$-Lipschitz.
	\end{proof}
	
	Let $\gamma$ be a $(c_S,M,L)$-admissible curve and let $E\subset [0,l_{\gamma}]^*$. Then we denote by
	\begin{equation*}
	T(\gamma\llcorner E):=\int_E\frac{1}{|b(\gamma(s))|}ds
	\end{equation*}
	the amount of time that a trajectory in $E$ spends on $\gamma(E)$ every period.
	In particular if $E=[0,l_\gamma]^*$ we simply denote by
	\begin{equation*}
	T(\gamma):=\int_0^{l_\gamma}\frac{1}{|b(\gamma(s))|}ds
	\end{equation*}
	the period of the trajectories on $\gamma([0,l_\gamma]^*)$.

	By means of Proposition \ref{P_cov} we can compare the above quantities on two different level sets.
\begin{lemma}\label{L_DeltaT}
	Let $\gamma_1,\gamma_2, A$ be as in Proposition \ref{P_cov}.
	Let moreover $F^1\subset D_1$ and $F^2=X_{1,2}(F^1)\subset D_2$, then 
	\begin{equation*}
	|T(\gamma_1\llcorner F^1)-T(\gamma_2\llcorner F^2)|\le \frac{1}{c_S^2}|Db|(A).
	\end{equation*}
	\end{lemma}
	\begin{proof}
	It holds 
	\begin{equation*}
	T(\gamma_1\llcorner F^1)= \sum_{j=1}^{N-1}T(\gamma_1\llcorner F^1_j),
	\end{equation*}
	where $F^1_j:=F^1\cap(s_{1,j}^-,s_{1,j}^+)$. We claim that 
	\begin{equation}\label{E_610}
	|T(\gamma_1\llcorner F^1_j)- T(\gamma_2\llcorner X_{1,2}(F^1_j))|\le \frac{1}{c_S^2}|Db|(E_j),
	\end{equation}
	where $E_j$ is defined in \eqref{E_def_Ej}.	
	By \eqref{E_Y_{1,j}} it holds
	\begin{equation*}
	\begin{split}
	T(\gamma_1\llcorner F^1_j)=&\int_{F^1_j}\frac{1}{|b(\gamma_1(s))|}ds \\
	= & \int_{Y_{1,j}(F^1_j)} \frac{1}{|b(\gamma_1(Y_{1,j}^{-1}(y)))|} \cdot \frac{1}{|Y_{1,j}'(y)|}dy \\
	= & \int_{Y_{1,j}(F^1_j)} \frac{dy}{|b(\gamma_1(Y_{1,j}^{-1}(y)))\cdot e_j|}.
	\end{split}
	\end{equation*}
	Since $Y_{1,j}(F^1_j)=Y_{2,j}(X_{1,2}(F^1_j))$ it holds
	\begin{equation*}
	T(\gamma_2\llcorner X_{1,2}(F^1_j))= \int_{Y_{1,j}(F^1_j)} \frac{dy}{|b(\gamma_2(X_{1,2}(Y_{1,j}^{-1}(y))))\cdot e_j|}.
	\end{equation*}
	We are in position to apply Lemma \ref{L_Dg} and we get
	\begin{equation*}
	\begin{split}
	|T(\gamma_1\llcorner F^1_j)-T(\gamma_2\llcorner X_{1,2}(F^1_j))| \le &  
	 \int_{Y_{1,j}(F^1_j)}\left| \frac{1}{|b(\gamma_1(Y_{1,j}^{-1}(y)))\cdot e_j|} - \frac{1}{|b(\gamma_2(X_{1,2}(Y_{1,j}^{-1}(y))))\cdot e_j|}\right|dy \\
	\le & \frac{1}{c_S^2} \int_{Y_{1,j}(F^1_j)}\left|  b(\gamma_1(Y_{1,j}^{-1}(y)))\cdot e_j  - b(\gamma_2(X_{1,2}(Y_{1,j}^{-1}(y))))\cdot e_j \right| dy \\
	\le & \frac{1}{c_S^2} |Db|(E_j)
	\end{split}
	\end{equation*}
	and this proves \eqref{E_610}.
	Since the sets $E_j\subset A$ are pairwise disjoint and in particular $X_{1,2}$ is one to one, it holds
	\begin{equation*}
	\begin{split}
	|T(\gamma_1\llcorner F^1)-T(\gamma_2\llcorner F^2)| = &~ \left| \sum_{j=1}^{N-1}T(\gamma_1\llcorner F^1_j)- T (\gamma_2(X_{1,2}(F^1_j))) \right| \\
	\le &~ \frac{1}{c_S^2}\sum_{j=1}^{N-1}|Db|(E_j) \\
	\le &~ \frac{1}{c_S^2}|Db|(A). \qedhere
	\end{split}
	\end{equation*}
	\end{proof}

%
%
	
	In the following proposition we compare the evolutions of two points starting from different trajectories in the same setting as in Proposition \ref{P_cov}. 
	
	\begin{proposition}\label{P_Lip}
	Let $\gamma_1,\gamma_2$ and $A$ be as in the statement of Proposition \ref{P_cov}. Then there exist  $\tilde c_3=\tilde c_3(c_S,M,L,\|b\|_{L^\infty})>0$, $\tilde c_4=\tilde c_4(\|b\|_{L^\infty}, c_S,L)>0$ and $\tilde c_5=\tilde c_5(\|b\|_{L^\infty}, c_S,L)>0$ such that for every $s_1\in [0,l_{\gamma_1}]^*$, $s_2\in [0,l_{\gamma_2}]^*$ and every $t>0$ it holds
	\begin{equation}\label{E_claim_Lip}
	|X(t,\gamma_1(s_1))-X(t,\gamma_2(s_2))|\le \tilde c_3\left(1+\frac{t}{T(\gamma_{1})}\right)|h_1-h_2|+ \tilde c_4 \left(1+\frac{t}{T(\gamma_{1})}\right)|Db|(A)+ \tilde c_5|\gamma_1(s_1)-\gamma_2(s_2)|.
	\end{equation}
	\end{proposition}
	\begin{proof}
	Denote by $x_1=\gamma_1(s_1)$ and $x_2=\gamma_2(s_2)$. Moreover let $s_1'\in D_1$ be such that $|s_1'-s_1|\le l_{\gamma_1}-\L^1(D_1)$ and set $x_1'=\gamma_1(s_1')$.
	By the triangle inequality
	\begin{equation*}
	\begin{split}
	|X(t,\gamma_1(s_1))-X(t,\gamma_2(s_2))|\le &|X(t,\gamma_1(s_1))-X(t,\gamma_1(s_1'))| + |X(t,\gamma_1(s_1'))-X(t,\gamma_2(X_{1,2}(s_1')))| \\&+ 
	|X(t,\gamma_2(X_{1,2}(s_1')))-X(t,\gamma_2(s_2))|.
	\end{split}
	\end{equation*}
	By Lemma \ref{L_same} and \eqref{E_est_h_1} we can estimate the first term:
	\begin{equation}\label{E_est_1st}
	\begin{split}
	|X(t,\gamma_1(s_1))-X(t,\gamma_1(s_1'))|\le &~ \frac{\|b\|_{L^\infty}^2}{c_S}L|\gamma_1(s_1)-\gamma_1(s_1')| \\
	\le &~  \frac{\|b\|_{L^\infty}^3}{c_S}L|s_1-s_1'| \\
	\le &~  \frac{\|b\|_{L^\infty}^3}{c_S}L(l_{\gamma_1}-|D_1|) \\
	\le &~  \frac{\|b\|_{L^\infty}^3}{c_S}L\left(\tilde c_1 |h_1-h_2| + \tilde c_2|Db|(A)\right).
	\end{split}
	\end{equation}
	The third term can be estimated in a similar way:
	\begin{equation*}
	|X(t,\gamma_2(X_{1,2}(s_1')))-X(t,\gamma_2(s_2))|\le  \frac{\|b\|_{L^\infty}^2}{c_S}L|\gamma_2(X_{1,2}(s_1'))-\gamma_2(s_2)|.
	\end{equation*}
	Again we can split
	\begin{equation*}
	|\gamma_2(X_{1,2}(s_1'))-\gamma_2(s_2)|\le |\gamma_2(X_{1,2}(s_1'))- \gamma_1(s_1')|+| \gamma_1(s_1')- \gamma_1(s_1)| 
	+ |\gamma_1(s_1)-\gamma_2(s_2)|.
	\end{equation*}
	By \eqref{E_max_dist} it holds
	\begin{equation}\label{E_est_1}
	 |\gamma_2(X_{1,2}(s_1'))- \gamma_1(s_1')|\le \frac{2\sqrt 2}{c_S}|h_1-h_2|.
	\end{equation}
	By \eqref{E_est_1} and estimating $|\gamma_1(s_1)-\gamma_1(s_1')|$ as in \eqref{E_est_1st} we get
	\begin{equation*}
	|\gamma_2(X_{1,2}(s_1'))-\gamma_2(s_2)|\le \left(\frac{2\sqrt 2}{c_S}+\|b\|_{L^\infty} \tilde c_1\right)|h_1-h_2| + \|b\|_{L^\infty} \tilde c_2|Db|(A)+ |\gamma_1(s_1)-\gamma_2(s_2)|
	\end{equation*}
	so that 
	\begin{equation}\label{E_613}
	\begin{split}
	|X(t,&\gamma_2(X_{1,2}(s_1')))-X(t,\gamma_2(s_2))|\\  & \le \frac{\|b\|_{L^\infty}^2}{c_S}L \left(\left(\frac{2\sqrt 2}{c_S}+\|b\|_{L^\infty} \tilde c_1\right)|h_1-h_2| +\|b\|_{L^\infty} \tilde c_2 |Db|(A) +  |\gamma_1(s_1)-\gamma_2(s_2)|\right).
	\end{split}
	\end{equation}
	
	It remains to estimate $|X(t,\gamma_1(s_1'))-X(t,\gamma_2(X_{1,2}(s_1')))|$.	
	Let
	\begin{equation*}
	t_1:=\max\{t'\le t:X(t',\gamma_1(s_1'))\in\gamma_1(D_1)\}
	\end{equation*}
	and $\bar s_1\in D_1$ be such that $X(t_1,\gamma_1( s_1'))=\gamma_1(\bar s_1)$.
	Since $\gamma_1$ is $(c_S,M)$-admissible, by \eqref{E_est_h_1} it holds $|t-t_1|\le \frac{l_{\gamma_1}-|D_1|}{c_S}$.
	Moreover, let $K\in \N$ and $t_1'\in [0,T(h_1))$ be such that $t_1=KT(h_1)+t_1'$ and let $t_2\in [KT(h_2),(K+1)T(h_2))=KT(\gamma_2)+t_2'$ be the unique value for which
	\begin{equation*}
	X(t_2,\gamma_2(X_{1,2}(s_1')))=\gamma_2(X_{1,2}(\bar s_1)).
	\end{equation*}
	Now we estimate
	\begin{equation}\label{E_triang}
	\begin{split}
	|X(t,\gamma_1(s_1'))-&X(t,\gamma_2(X_{1,2}(s_1')))|\le 
	|X(t,\gamma_1(s_1'))-X(t_1,\gamma_1(s_1'))|+|X(t_1,\gamma_1(s_1'))-\gamma_2(X_{1,2}(\bar s_1))| \\
	&+|\gamma_2(X_{1,2}(\bar s_1))-X(t,\gamma_2(X_{1,2}(s_1')))|.
	\end{split}
	\end{equation}
	The first term is easily estimated by
	\begin{equation}\label{E_619}
	|X(t,\gamma_1(s_1'))-X(t_1,\gamma_1(s_1'))|\le \|b\|_{L^\infty}|t-t_1|\le \frac{\|b\|_{L^\infty}}{c_S}\big( \tilde c_1 |h_1-h_2| + \tilde c_2 |Db|(A)\big).
	\end{equation}
	By \eqref{E_max_dist} we have
	\begin{equation}\label{E_620}
	\begin{split}
	|X(t_1,\gamma_1(s_1'))-\gamma_2(X_{1,2}(\bar s_1))|= &|\gamma_1(\bar s_1)-\gamma_2(X_{1,2}(\bar s_1))| \\
	\le & \sup_{D_1}|\gamma_1-\gamma_2\circ X_{1,2}| \\
	\le &  \frac{2\sqrt 2}{c_S}|h_1-h_2|.
	\end{split}
	\end{equation}
	For the third term on the right hand side of \eqref{E_triang} we have
	\begin{equation}\label{E_est_third}
	\begin{split}
	|\gamma_2(X_{1,2}(\bar s_1))-X(t,\gamma_2(X_{1,2}(s_1')))| =& |X( t_2,\gamma_2(X_{1,2}(s_1')))-X(t,\gamma_2(X_{1,2}(s_1')))| \\
	\le & \|b\|_{L^\infty} |t-t_2|\\
	\le & \|b\|_{L^\infty} (|t-t_1|+|t_1- t_2|)\\
	\le & \|b\|_{L^\infty}  \frac{l_{\gamma_1}-|D_1|}{c_S} +   \|b\|_{L^\infty}|t_1-t_2| \\
	\le & \frac{\|b\|_{L^\infty}}{c_S}\big( \tilde c_1 |h_1-h_2| + \tilde c_2 |Db|(A)\big)  +   \|b\|_{L^\infty}|t_1-t_2|.
	\end{split}
	\end{equation}
	So it remains to estimate $|t_1-t_2|$.
	\begin{equation*}
	\begin{split}
	|t_1-t_2| = & |KT(\gamma_1)+t_1'-KT(\gamma_2)-t_2'| \\
	\le & K|T(\gamma_1)-T(\gamma_2)| + |t_1'-t_2'|.
	\end{split}
	\end{equation*}
	First we estimate
	\begin{equation*}
	\begin{split}
	|T(\gamma_1)-T(\gamma_2)| = & \left|T(\gamma_1\llcorner D_1)+T(\gamma_1\llcorner ([0,l_{\gamma_1}]\setminus D_1))-
	T(\gamma_2\llcorner D_2)-T(\gamma_2\llcorner ([0,l_{\gamma_2}]\setminus D_2))\right| \\
	\le & \left| T(\gamma_1\llcorner D_1) -T(\gamma_2\llcorner D_2) \right| + T(\gamma_1\llcorner ([0,l_{\gamma_1}]\setminus D_1) + T(\gamma_2\llcorner ([0,l_{\gamma_2}]\setminus D_2).
	\end{split}
	\end{equation*}
	Since $\gamma_1,\gamma_2$ are $(c_S,M,L)$-admissible and \eqref{E_est_h_1}, then 
	\begin{equation*}
	\begin{split}
	T(\gamma_1\llcorner([0,l_{\gamma_1}]\setminus D_1))\le& \frac{1}{c_S}\L^1([0,l_{\gamma_1}]\setminus D_1))\le \frac{1}{c_S}\big(\tilde c_1 |h_1-h_2| + \tilde c_2 |Db|(A) \big), \\
	T(\gamma_2\llcorner([0,l_{\gamma_2}]\setminus D_2))\le& \frac{1}{c_S}\L^1([0,l_{\gamma_2}]\setminus D_2))\le  \frac{1}{c_S}\big(\tilde c_1 |h_1-h_2| + \tilde c_2 |Db|(A) \big).
	\end{split}
	\end{equation*}	
	Moreover by Lemma \ref{L_DeltaT} it holds
	\begin{equation*}
	|T(\gamma_1\llcorner D_1)-T(\gamma_2\llcorner D_2)|\le \frac{1}{c_S^2} |Db|(A).
	\end{equation*}
	Similarly we can estimate $|t_1'-t_2'|$. 
	We have the following:
	\begin{equation*}
	t_1'=T(\gamma_1\llcorner [s_1',\bar s_1]), \qquad \mbox{and} \qquad t_2'=T(\gamma_2\llcorner [X_{1,2}(s_1'),X_{1,2}(\bar s_1)]).
	\end{equation*}
	Since $X_{1,2}$ is monotone (Proposition \ref{P_cov}) it holds
	\begin{equation*}
	X_{1,2}((s_1',\bar s_1)\cap D_1)=(X_{1,2}(s_1'),X_{1,2}(\bar s_1))\cap D_2
	\end{equation*}
	so that relying again upon Lemma \ref{L_DeltaT}, it holds
	\begin{equation*}
	\begin{split}
	|t_1'-t_2'| \le & \left| T(\gamma_1\llcorner((s_1',\bar s_1)\cap D_1))-T(\gamma_2\llcorner((X_{1,2}(s_1'),X_{1,2}(\bar s_1))\cap D_2))\right| \\
	& +T(\gamma_1\llcorner((s_1',\bar s_1)\setminus D_1)) + T(\gamma_2\llcorner((X_{1,2}(s_1'),X_{1,2}(\bar s_1))\setminus D_2)) \\
	\le & \frac{1}{c_S^2}|Db|(A_{h_1,h_2}) + \frac{1}{c_S}(l_{\gamma_1}-\L^1(D_1) + l_{\gamma_2}-\L^1(D_2)) \\
	\le & \frac{2\tilde c_1}{c_S}|h_1-h_2| + \left(\frac{1}{c_S^2}+ \frac{2\tilde c_2}{c_S}\right)|Db|(A).
	\end{split}
	\end{equation*}	
	Finally
	\begin{equation}\label{E_est_Deltat}
	|t_1- t_2|\le (K+1)\left[ \frac{2\tilde c_1}{c_S}|h_1-h_2| + \left(\frac{1}{c_S^2}+ \frac{2\tilde c_2}{c_S}\right) |Db|(A) \right].
	\end{equation}
	By \eqref{E_est_Deltat} we can estimate in \eqref{E_est_third}
	\begin{equation}\label{E_all_1}
	|\gamma_2(X_{1,2}(\bar s_1))-X(t,\gamma_2(X_{1,2}(s_1')))| \le  \frac{\|b\|_{L^\infty}}{c_S}  \left[ (2K+3)\tilde c_1 |h_1-h_2|+ \left( (K+1) \left( 2\tilde c_2 + \frac{1}{c_S}\right) + \tilde c_2\right) |Db|(A)\right].
	\end{equation}
	Plugging \eqref{E_all_1}, \eqref{E_619} and \eqref{E_620} into \eqref{E_triang} we get 
	\begin{equation}\label{E_all_2}
	\begin{split}
	|X(t,\gamma_1(s_1'))-X(t,\gamma_2(X_{1,2}(s_1')))|\le & \left[ \frac{\|b\|_{L^\infty}}{c_S}2(K+2)\tilde c_1+ \frac{2\sqrt 2}{c_S}\right] |h_1-h_2| \\
	&~ + \frac{\|b\|_{L^\infty}}{c_S}\left[(K+1)\left(2\tilde c_2 + \frac{1}{c_S}\right)+2\tilde c_2 \right]|Db|(A).
	\end{split}
	\end{equation}
	Finally from \eqref{E_all_2}, \eqref{E_est_1st} and \eqref{E_613} it holds
	\begin{equation}\label{E_end}
	|X(t,\gamma_1(s_1))-X(t,\gamma_2(s_2))|\le \tilde c_3(1+K)|h_1-h_2|+ \tilde c_4 (1+K)|Db|(A)+ \tilde c_5|\gamma_1(s_1)-\gamma_2(s_2)|,
	\end{equation}
	for some constants $\tilde c_3=\tilde c_3(c_S,M,L,\|b\|_{L^\infty})>0$, $\tilde c_4=\tilde c_4(\|b\|_{L^\infty}, c_S,L)>0$ and $\tilde c_5=\tilde c_5(\|b\|_{L^\infty}, c_S,L)>0$.
	By definition of $K$ it holds $K\le t/T(\gamma_1)$, therefore from \eqref{E_end} it immediately follows \eqref{E_claim_Lip} and this concludes the proof.
	\end{proof}
	
%
%
%

	Building on Lemma \ref{L_same} and Proposition \ref{P_Lip} we are eventually in position to prove our main result.
	
	\begin{theorem}\label{T_Lip}
	Let $H$ be a Lipschitz Hamiltonian with compact support and assume that $b=\nabla^\perp H \in \BV(\R^2,\R^2)$. Denote by $X$ the regular Lagrangian flow of $b$.
	Then for every $\e>0$ there exists $C=C(\e,H)>0$ and $B\subset \R^2$ with $|B|\le \e$ such that for every $t\ge 0$ it holds
	\begin{equation*}
	\Lip (X(t)\llcorner (\R^2\setminus B))\le C(1+t).
	\end{equation*}
	\end{theorem}
	\begin{proof}
	Consider the decomposition of $H$ as in Theorem \ref{T_ABC}.
	By Corollary \ref{C_ABC} we deduce that the regular Lagrangian flows associated to $H$ and to $H_i$ coincide for every $t\ge 0$ and for $\mathscr L^2$-a.e. $x \in A_i$.
	
	Let $N\in \N$ be such that
	\begin{equation*}
	\left| \bigcup_{i>N}A_i \right| <\frac{\e}{4}.
	\end{equation*}
	Denoting by $X_N$ the regular Lagrangian flow associated to $b_N:= \nabla^\perp \sum_{i=1}^NH_i$ it follows from the previous observation that there exists $B_1\subset \R^2$
	such that $|B_1|\le \e/4$ and for every $t\ge 0$ and every $x\in \R^2\setminus B_1$ it holds $X(t,x)=X_N(t,x)$.
	
	For $i=1,\ldots, N$ we denote by $G_i:\R\to [0,+\infty]$ the maximal function of the measure ${H_i}_\sharp |Db|$.  
	For $M>0$ we define 
	\begin{equation*}
	E^i_M:= \bigcup_{(h,j)\in I^i_M}C^j_h,
	\end{equation*}
	where $C^j_h$ are defined in Theorem \ref{T_Bourgain} and the set $I^i_M$ is the set of pairs $(h,j)$ such that $h\in R'$, $\bar i (h,j)=i$, the curve $\gamma^j_h$ is 
	$(1/M,M,M)$-admissible, the period $T(\gamma^j_h)\ge 1/M$ and $G_i(\bar h(h,j))\le M$. Notice that Property (4) in Theorem \ref{T_Bourgain} implies that the curves $\gamma^j_h$ have no cusps so that by 
	Lemma \ref{L_Lip_inverse} we have that for every $i=1,\ldots,N$ it holds 
	\begin{equation*}
	\left|A_i\setminus \bigcup_{M>0} E^i_M\right|=0,
	\end{equation*}
	therefore there exists $M_i$ such that $|A_i\setminus E^i_{M_i}|\le \e/4N$.
	
	In the following we will consider a representative of the flow $X$ as in Remark \ref{R_representative}.
	
	\emph{Claim}. For every $i=1,\ldots,N$ there exists $C_{M_i}>0$ such that for every $x_1,x_2 \in E^i_{M_i}$ and every $t\ge 0$ it holds
	\begin{equation}\label{E_claim}
	|X(t,x_1)-X(t,x_2)|\le C_{M_i}(1+t)|x_1-x_2|.
	\end{equation}
	
	By definition of $E^i_{M_i}$ there exist $(h_1,j_1),(h_2,j_2)\in \R\times \N$ such that $x_1 \in C^{j_1}_{h_1}$, $x_2 \in C^{j_2}_{h_2}$ and $\bar i(h_1,j_1)=\bar i(h_2,j_2)=i$. 
	Denote by $\gamma_1=\gamma_{h_1}^{j_1}$ and $\gamma_2=\gamma_{h_2}^{j_2}$
	By the monotonicity of the Hamiltonian $H_i$ we are in position to apply Proposition \ref{P_Lip} to the curves $\gamma_1,\gamma_2$ (see Remark \ref{R_signs}): in particular
	there exists a uniform constant $\tilde C_{M_i}>0$ such that for every $x_1,x_2 \in E^i_{M_i}$ and for every $t\ge 0$ it holds
	\begin{equation}\label{E_444}
	|X(t,x_1)-X(t,x_2)|\le \tilde C_{M_i}\left(|h_1-h_2| + |Db|(A) + |x_1-x_2|\right)(1+t).
	\end{equation}
	It trivially holds $|h_1-h_2|\le \|b\|_{L^\infty}|x_1-x_2|$ moreover by the monotonicity of $H_i$ we have that $A=H_i^{-1}(\bar h(h_1,j_1),\bar h (h_2,j_2))$, where we assumed without loss of generality that $\bar h(h_1,j_1)<\bar h (h_2,j_2)$. Therefore it holds
	\begin{equation}\label{E_445}
	|Db|(A)= H_i\sharp |Db| ((h_1,h_2)) \le G_i(\bar h(h_1,j_1))|\bar h(h_2,j_2)-\bar h(h_1,j_1)|\le M\|b\|_{L^\infty}|x_1-x_2|.
	\end{equation}
	By \eqref{E_444} and \eqref{E_445} it follows \eqref{E_claim}.
	
	In order to compare the evolutions of trajectories not belonging to the same set $E^i_{M_i}$ we provide an elementary compactness argument.
	For any $i=1,\ldots,N$ let $K_i\subset E^i_{M_i}$ be a compact set such that $ |E^i_{M_i}\setminus K_i|\le \e/4(N+1)$ and let 
	\begin{equation*}
	\tilde K_i:=\{x\in \R^2: \exists x'\in K_i, t\ge 0 \mbox{ s. t. }X(t,x')=x\}. 
	\end{equation*}
	It follows from \eqref{E_claim} that the sets $\tilde K_i$ are compact.
	Moreover let $A_0=\R^2\setminus \bigcup_{i\in \N}A_i$ and let $\tilde K_0\subset A_0$ be a closed set such that $|A_0\setminus \tilde K_0|\le \e/4(N+1)$. 
	By Theorem \ref{T_ABC} we have $\L^1(H(A_0))=0$, therefore $b(x)=0$ for $\L^2$-a.e. $x\in A_0$. In particular for $\L^2$-a.e. $x\in A_0$ and for every $t\ge 0$ it holds
	$X(t,x)=x$.

	Setting $B=\R^2\setminus \bigcup_{i=0}^N\tilde K_i$, we have by construction that $|B|\le \e$. 
	We finally check that there exists $C>0$ such that for every $x_1,x_2\in \bigcup_{i=0}^N\tilde K_i$ it holds 
	\begin{equation}\label{E_thesis}
	|X(t,x_2)-X(t,x_1)|\le C(1+t)|x_1-x_2|.
	\end{equation}
	If $x_1 \in  \tilde K_i$ and $x_2 \in  \tilde K_j$ for some $i\ne j$, then 
	\begin{equation*}
	|X(t,x_1)-X(t,x_2)|\le |x_1-x_2| + 2\|b\|_{L^\infty} t \le \left(1+\frac{2\|b\|_{L^\infty} t}{\delta}\right)|x_1-x_2|,
	\end{equation*}
	where
	\begin{equation*}
	\delta:= \min_{i,j\in 0,\ldots,N}\dist(\tilde K_i,\tilde K_j).
	\end{equation*}
	The case $x_1,x_2\in \tilde K_0$ is trivial and if $x_1,x_2\in\tilde K_i$ for some $i=1,\ldots,N$, then \eqref{E_thesis} follows from \eqref{E_claim}.
	\end{proof}
	
\section{An application to mixing}\label{S_mixing}
	In this section we deduce lower bounds for the two notions of geometric and analytical mixing by means of the Lusin-Lipschitz regularity estimate on the flow obtained in Theorem
	\ref{T_Lip}.
	
	We consider the setting of the previous section and we additionally assume that $B_1$ is an invariant region for the flow $X$ of the vector field $b$.
	Notice that for every cycle $C$ as in Theorem \ref{T_Bourgain}, we have that $\Int C$ is an invariant region for $X$, so that requiring that $B_1$ is invariant is not a strong
	restriction. 
	 We moreover consider \eqref{E_CE_Cauchy} with initial datum of the form 
	$u_0=\chi_A-\chi_{B_1\setminus A}$ for some set $A\subset B_1$ with $\L^2(A)=\L^2(B_1)/2$.
	We will deal with the two following notions of mixing.

	\begin{definition}\label{D_geom}
	Let $k\in (0,1/2)$. We say that the \emph{geometric mixing scale} of $u(t)$ with accuracy parameter $k$ in $B_1$ is the infimum $\mathcal G(u(t))$ of $\delta>0$ 
	such that for every $x\in B_1$ it hold
	\begin{equation*}
	\frac{\L^2(B(x,\delta)\cap u(t)^{-1}(\{1\}))}{\L^2(B(x,\delta))}<1-k \qquad \mbox{and} \qquad \frac{\L^2(B(x,\delta)\cap u(t)^{-1}(\{-1\}))}{\L^2(B(x,\delta))}<1-k.
	\end{equation*}
	\end{definition}
	
	\begin{definition}
	We say that the \emph{functional mixing scale} of $u(t)$ is $\|u(t)\|_{\dot H^{-1}(B_1)}$, where
	\begin{equation*}
	\|u\|_{\dot H^{-1}(B_1)}:=\sup \left\{\int_{B_1}u\phi dx: \|\nabla \phi\|_{L^2}\le 1 \right\}.
	\end{equation*}
	\end{definition}
	
	In order to explicit the dependence of the lower bounds estimates of the mixing with respect to the initial datum we follow \cite{IKX_mixing}:
	given $\alpha_0\in (0,\L^2(B_1))$, $k_0\in (0,1/2)$ and $A\subset B_1$ as above, we denote by
	\begin{equation}\label{E_def_r0}
	\bar r_0(A):= \sup\{r_0:  \L^2(G_0(r_0))\ge \alpha_0\},
	\end{equation}
	where $G_0(r_0)$ is defined by
	\begin{equation*}
	G_0(r_0):=\left\{x\in B_1: \frac{|B(x,r_0)\cap A|}{|B(x,r_0)|}>1-k_0\right\}.
	\end{equation*}
	Notice that a simple application of the Lebesgue differentiation theorem shows that for every $A\subset B_1$ as above it holds $\bar r_0(A)>0$.

	\begin{proposition}\label{P_mixing}
	Let $\bar t, \delta>0$, $k_0,k_1\in (0,\frac{1}{2})$, $A\subset B_1$ such that $\L^2(A)=\L^2(B_1)/2$ and denote by 
	\begin{equation*}
	F_{\bar t}(\delta):=\left\{x\in B_1: \frac{\L^2(B(x,\delta)\cap X(\bar t, A))}{\L^2(B_\delta)}<1-k_1 \right\}.
	\end{equation*}
	Let $\alpha_1>\L^2(B_1\setminus F_{\bar t})$ and suppose that there exists $\e>0$ such that
	\begin{equation}\label{E_parameters}
	\left(\frac{k_1}{10}\left(\frac{1}{4}-k_0\right)-k_0\right)\alpha_0 - \alpha_1k_1 - \e k_1-10 \e>0,
	\end{equation}
	where $\alpha_0>0$ as in \eqref{E_def_r0}.
	Then
	\begin{equation}\label{E_r0delta}
	\delta \ge \frac{\bar r_0(A)}{2C(1+\bar t)},
	\end{equation}
	where $C=C(\e,b)$ is given by Theorem \ref{T_Lip}.
	\end{proposition}
	\begin{remark}
	Notice that for any $k_1\in (0,1/2)$ and $\alpha_0>0$ given, there exist $\alpha_1,k_0,\e>0$ small enough such that \eqref{E_parameters} is satisfied.
	\end{remark}
	\begin{proof}
	Let $r_0<\bar r_0(A)$ be such that $\L^2(G_0(r_0))\ge \alpha_0$.
	By the Vitali's covering lemma there exist $l\in \N$ and a pairwise disjoint family of balls $(B(x_i,r_0))_{i=1}^l$ such that $x_i\in G_0(r_0)$ and
	\begin{equation}\label{E_dim1}
	l\L^2(B_{r_0})=\L^2\left(\bigcup_{i=1}^lB(x_i,r_0)\right) \ge \frac{\alpha_0}{10}.
	\end{equation}
	Denote by
	\begin{equation*}
	Z_1:=X\left(\bar t,  \bigcup_{i=1}^lB(x_i,r_0)\setminus A \right).
	\end{equation*}
	Since the flow preserves the Lebesgue measure, we have
	\begin{equation}\label{E_dim2}
	\L^2(Z_1)=\L^2\left(  \bigcup_{i=1}^lB(x_i,r_0))\setminus A\right) \le k_0 l\L^2(B_{r_0}).
	\end{equation}
	Given $\e>0$ as in the statement, denote by $Z_2\subset B_1$ a set such that $\L^2(Z_2)\le \e$ and $X(-\bar t)\llcorner (B_1\setminus Z_2)$ is $C(1+\bar t)$-Lipschitz provided by Theorem \ref{T_Lip}.
	We denote by
	\begin{equation*}
	C_2:=\bigcup_{i=1}^lB\left(x_i,\frac{r_0}{2}\right)
	\end{equation*}
	By assumption $\L^2(B_1\setminus F_{\bar t}(\delta))<\alpha_1$. Then
	\begin{equation*}
	\begin{split}
	\L^2(X(\bar t, C_2)\cap Z_1^c\cap Z_2^c\cap F_{\bar t}(\delta)) \ge &~ \L^2(C_2) - \L^2(B_1\setminus F_{\bar t}(\delta))-\L^2(Z_1)-\L^2(Z_2) \\
	\ge &~ \frac{l}{4}\L^2(B_{r_0}) - \alpha_1-k_0 l\L^2(B_{r_0})-\e.
	\end{split}
	\end{equation*}
	Again by Vitali's covering lemma for the set $X(\bar t, C_2)\cap Z_1^c\cap Z_2^c\cap F_{\bar t}$, there exists a pairwise disjoint family of balls $(B(\tilde x_i,\delta))_{i=1}^{\tilde l}$
	such that $\tilde x_i\in  X(\bar t, C_2)\cap Z_1^c\cap Z_2^c\cap F_{\bar t}$ and
	\begin{equation}\label{E_dim3}
	\tilde l\L^2(B_\delta) \ge \frac{1}{10}\left( \frac{l}{4}\L^2(B_{r_0}) - \alpha_1-k_0 l\L^2(B_{r_0})-\e\right).
	\end{equation}
	By definition of $F_{\bar t}$ it holds
	\begin{equation*}
	\L^2\left(\bigcup_{i=1}^{\tilde l}B(\tilde x_i,\delta)\cap X(\bar t,A^c)\right) \ge k_1\tilde l \L^2(B_\delta).
	\end{equation*}
	and therefore 
	\begin{equation}\label{E_dim4}
	\begin{split}
	\L^2\left(\bigcup_{i=1}^{\tilde l}B(\tilde x_i,\delta)\cap X(\bar t,A^c)\cap Z_2^c\right)\ge &~\L^2\left(\bigcup_{i=1}^{\tilde l}B(\tilde x_i,\delta)\cap X(\bar t,A^c)\right)  - \L^2(Z_2) \\
	\ge &~ k_1\tilde l \L^2(B_\delta)- \e.
	\end{split}
	\end{equation}
	If 
	\begin{equation}\label{E_ass}
	\L^2\left(\bigcup_{j=1}^{\tilde l}B(\tilde x_j,\delta)\cap X(\bar t,A^c)\cap Z_2^c\right)> \L^2(Z_1),
	\end{equation}
	then
	\begin{equation*}
	\left(\bigcup_{j=1}^{\tilde l}B(\tilde x_j,\delta)\cap X(\bar t,A^c)\cap Z_2^c\right) \setminus Z_1 \ne \emptyset,
	\end{equation*}
	i.e. there exist $j\in 1,\ldots, \tilde l$ and $y \in B(\tilde x_j,\delta)\cap Z_2^c$ such that $X(-\bar t,y)\notin B(x_i,r_0)$ for every $i=1,\ldots, l$. 
	Since $\tilde x_j\in X(\bar t, C_2)\cap Z_2^c$, then 
	\begin{equation*}
	\frac{r_0}{2\delta}\le \frac{|X(-\bar t, \tilde x_j)-X(-\bar t,y)|}{|\tilde x_j-y|}\le C(1+\bar t),
	\end{equation*}
	which immediately implies \eqref{E_r0delta}.
	It is a straightforward computation to check that \eqref{E_ass} is granted by \eqref{E_parameters} 
	by combining \eqref{E_dim1}, \eqref{E_dim2}, \eqref{E_dim3} and \eqref{E_dim4} and this concludes the proof.
	\end{proof}

	\begin{lemma}\label{L_isoperimetric}
	Let $\alpha_0$, $k_0\in (0,1/2)$ and $A\subset B_1$ be such that $\L^2(A)=|B_1|/2$ and let $\bar r_0(A)$ be defined by \eqref{E_def_r0}.
	Assume moreover that 
	\begin{equation*}
	\beta:= \frac{\L^2(B_1)}{2}-\alpha_0>0.
	\end{equation*}
	Then there exists $c=c(\beta, k_0)$ such that
	\begin{equation*}
	\Per(A,B_1) \ge \frac{c}{\bar r_0(A)}.
	\end{equation*}
	\end{lemma}
	\begin{proof}
	Let $\bar r \in(\bar r_0(A),1)$, then 
	\begin{equation*}
	\L^2(A\cap G_0(\bar r)^c)>\frac{\L^2(B_1)}{2}-\alpha_0=:\beta>0.
	\end{equation*}
	 By Besicovitch's covering theorem, there exists an absolute constant $\underline c>0$ and a pairwise disjoint family of balls 
	$(B(x_i,\bar r))_{i=1}^l$ with $x_i\in A\cap G_0(\bar r)^c$ and
	\begin{equation*}
	\L^2\left(\bigcup_{i=1}^lB(x_1,\bar r) \cap A\cap G_0(\bar r)^c\right)\ge \frac{\beta}{\underline c}.
	\end{equation*}
	Given $\tilde k_1\in (0,1/2)$ by elementary combinatorics, there exist at least $N$ indexes $i_1,\ldots,i_N \in\{1,\ldots, l\}$ such that
	\begin{equation}\label{E_est_density}
	\frac{\L^2(B(x_i,\bar r)\cap A)}{\L^2(B_{\bar r})}\ge \tilde k_1, \qquad \mbox{with} \qquad N= \left\lceil \frac{\frac{\beta}{\underline c}-l\tilde k_1\L^2(B_{\bar r})}{(1-\tilde k_1-k_0)\L^2(B_{\bar r})}\right\rceil.
	\end{equation}
	Choosing $\tilde k_1$ such that 
	\begin{equation*}
	l\tilde k_1\L^2(B_{\bar r})<\frac{\beta}{2\underline c},
	\end{equation*}
	we get the estimate
	\begin{equation}\label{E_est_N}
	N \ge  \left\lceil \frac{\beta}{2\underline c \L^2(B_{\bar r})} \right\rceil.
	\end{equation} 	Let $k:=\min\{k_0,\tilde k_1\}$. By the definition of $G_0$ and \eqref{E_est_density} for every $i\in \{i_1,\ldots, i_N\}$ it holds
	\begin{equation*}
	k\le \frac{\L^2(B(x_i,\bar r)\cap A)}{\L^2(B_{\bar r})}\le 1-k,
	\end{equation*}
	therefore by the relative isoperimetric inequality there exists a constant $c_1=c_1(k)$ such that for every  $i\in \{i_1,\ldots, i_N\}$ it holds
	\begin{equation*}
	\Per(A,(B(x_i,\bar r)))\ge c_1 \bar r.
	\end{equation*}
	Being the balls $B(x_i,\bar r)$ disjoint, we get by \eqref{E_est_N} that there exists $c=c(k,\beta)>0$ such that
	\begin{equation*}
	\Per(A,B_1)\ge \frac{c}{\bar r}. \qedhere
	\end{equation*}
	Since $\bar r>\bar r_0(A)$ is arbitrary this concludes the proof.
	\end{proof}
	
	\begin{corollary}\label{C_mix_g}
	Let $k\in (0,1/2)$ be the accuracy parameter as in Definition \ref{D_geom}. Then there exists $c_g=c_g(b,k)>0$ such that for every $t>0$
	\begin{equation}\label{E_mix_g}
	\mathcal G(u(t))\ge c_g\frac{\bar r_0(A)}{1+t}.
	\end{equation}
	Moreover there exists $\tilde c_g=\tilde c_g(b,k)>0$ such that for every $A$ as above and such that $\Per(A,B_1)<\infty$ and for every $t\ge 0$ it holds
	\begin{equation}\label{E_mix_g_per}
	\mathcal G(u(t))\ge \frac{\tilde c_g}{(1+t)\Per(A,B_1)}.
	\end{equation}
	\end{corollary}
	\begin{proof}
	Assume that $\mathcal G(u(t))<\delta$ for some $\delta>0$. 
	Then $\L^2(B_1\setminus F_t(\delta))=0$, therefore we are in position to apply Proposition \ref{P_mixing} with the following choice of parameters:
	\begin{equation}\label{E_choice}
	k_1=k, \quad \alpha_0=\frac{\L^2(B_1)}{4}, \quad k_0=\frac{k}{160},\quad \alpha_1=\frac{\alpha_0}{480} \quad\mbox{and} \quad \e=\frac{\alpha_0k}{480\cdot 11}.
	\end{equation}
	It is trivial to check that \eqref{E_parameters} holds. Therefore setting $c_g=(C(\e(k),b))^{-1}>0$, where $C$ is defined in Proposition \ref{P_Lip}, it holds
	\begin{equation*}
	\delta \ge c_g\frac{\bar r_0(A)}{1+t}.
	\end{equation*}
	Since $\delta > \mathcal G(u(t))$ is arbitrary, this proves \eqref{E_mix_g}.
	Finally by Lemma \ref{L_isoperimetric} it follows that if there exists a constant $\tilde c=\tilde c(k)>0$ such that 
	\begin{equation*}
	\bar r_0(A) \ge \frac{\tilde c}{\Per(A,B_1)}.
	\end{equation*}
	Therefore, setting $\tilde c_g=c_g\tilde c$ we get \eqref{E_mix_g_per}.
	\end{proof}

	\begin{corollary}
	Let $A\subset B_1$ with $\L^2(A)=\L^2(B_1)/2$ and let $u$ be the bounded weak solution to \eqref{E_CE_Cauchy} with $u_0=\chi_A-\chi_{B_1\setminus A}$. 
	Then there exists $c_a=c_a(b)$ such that for every $t\ge 0$
	\begin{equation}\label{E_mix_a}
	\|u(t)\|_{\dot H^{-1}(B_1)}\ge c_a\frac{\tilde r_0(A)}{1+t},
	\end{equation}
	where $\tilde r_0(A)$ is $\bar r_0(A)$ defined in \eqref{E_def_r0}, where we have chosen $k_0=\frac{1}{640}$ and $\alpha_0=\L^2(B_1)/4$.
	
	Moreover there exists $\tilde c_a=\tilde c_a(b)$ such that if $\Per(A,B_1)<\infty$, then for every $t\ge 0$ it holds
	\begin{equation}\label{E_mix_a_per}
	\|u(t)\|_{\dot H^{-1}(B_1)}\ge\frac{\tilde c_a}{(1+t)\Per(A,B_1)}.
	\end{equation}
	\end{corollary}
	\begin{proof}
	We consider the same choice of parameters as in \eqref{E_choice} and we additionally impose $k=1/4$.
	Let 
	\begin{equation}\label{E_def_delta}
	\delta= \frac{\tilde r_0(A)}{4C(1+t)},
	\end{equation}
	with the same constant $C$ as in Proposition \ref{P_mixing}. It follows by Proposition \ref{P_mixing} that 
	\begin{equation*}
	\L^2(B_1\setminus F_{t}(\delta)) \ge \alpha_1.
	\end{equation*}
	By definition of $F_t(\delta)$ and by Vitali's covering lemma there exists $l\in \N$ and a pairwise disjoint family of balls $(B(x_i,\delta))_{i=1}^l$ such that 
	$x_i\in B_1\setminus F_t(\delta)$ and $l\L^2(B_\delta)\ge \alpha_1/10$.
	Given $a,b>0$, denote by $\psi_{a,b}:[0,+\infty)\to [0,1]$ a smooth function such that 
	$\psi_{a,b}(r)=1$ if $r\le a$, $\psi_{a,b}(r)=0$ if $r\ge a+b$ and $|\psi'_{a,b}|(r)\le 2/b$ for every $r\ge 0$. 
	For every $i=1,\ldots,l$ we consider $\phi_i:\R^2\to \R$ defined by
	\begin{equation*}
	\phi_i(x):=\psi_{a,b}(|x-x_i|), \qquad \mbox{where}\qquad a=\delta\sqrt{\frac{3}{4}} \quad \mbox{and} \quad b=\delta -a.
	\end{equation*}
	By a straightforward computation there exists an absolute constant $\bar c$ such that
	\begin{equation}\label{E_536}
	\int_{B_1}|\nabla \phi_i|^2dx \le \bar c.
	\end{equation}
	Moreover, since $x_i \in B_1\setminus F_t(\delta)$ it follows that
	\begin{equation}\label{E_537}
	\begin{split}
	\int_{B_1}u(t)\phi_idx \ge&~  \L^2(B(x_i,a))-2\L^2(u(t)^{-1}(\{-1\})\cap B(x_i,\delta)) \\
	\ge &~ \L^2(B(x_i,a))-\frac{1}{2}\L^2(B(x_i,\delta)) \\
	= &~ \frac{\L^2(B_\delta)}{4}.	
	\end{split}
	\end{equation}
	Set $\phi=\sum_{i=1}^l\phi_i$. Since the balls $(B_i)_{i=1}^l$ are disjoint and they are all contained in $B_2$ it holds $l\L^2(B_\delta)\le \L^2(B_2)$, 
	therefore by \eqref{E_536} it follows that
	\begin{equation*}
	\left(\int_{B_1}|\nabla\phi|^2dx\right)^{1/2} \le \sqrt{l\bar c}\le \sqrt{\bar c}\sqrt{\frac{\L^2(B_2)}{\L^2(B_\delta)}} = \sqrt{\bar c}\frac{2}{\delta}.
	\end{equation*}
	By \eqref{E_537} instead it follows that 
	\begin{equation*}
	\int_{B_1}u(t)\phi_idx \ge\frac{l}{4}\L^2(B_\delta) \ge \frac{\alpha_1}{40},
	\end{equation*}
	so that there exists an absolute constant $c>0$ such that 
	\begin{equation*}
	\|u(t)\|_{\dot H^{-1}(B_1)} \ge c\delta
	\end{equation*}
	so that \eqref{E_mix_a} follows by the choice of $\delta$ in \eqref{E_def_delta}. The same argument as in the proof of Corollary \ref{C_mix_g} by means of 
	Lemma \ref{L_isoperimetric} shows \eqref{E_mix_a_per} and this concludes the proof.
	\end{proof}

\bibliographystyle{alpha}

\end{document}